\documentclass[a4paper,12pt]{amsart}

\pdfoutput=1 

\usepackage[top=3cm,bottom=2.5cm,outer=3cm,inner=2cm,marginpar=2.45cm]{geometry}
\usepackage[destlabel,final,colorlinks=true,urlcolor=RoyalBlue]{hyperref}
\usepackage[abbrev]{amsrefs}

\usepackage[french,ngerman,english]{babel}
\usepackage[utf8]{inputenc}
\usepackage[T1]{fontenc}

\babeltags{de = ngerman}
\babeltags{fr = french}

\usepackage{etoolbox} 

\usepackage[all,pdf]{xy}

\renewcommand{\objectstyle}{\displaystyle}

\usepackage[inline]{enumitem}
\usepackage{mathtools}  
\usepackage[pdftex,usenames,dvipsnames]{xcolor}
\usepackage{calc} 

\usepackage{upref}  
\usepackage{embrac} 

\usepackage[
  show-mario,  
]{marionotations}

\usepackage{bbm}     

\usepackage[Smaller]{cancel} 

\usepackage{breakurl}  

\usepackage{subfiles}

\newcommand{\defaultDimension}{n}

\newcommand{\defaultAmbientSpace}{X}

\newcommand{\defaultlcIndex}{\sigma}

\newcommand{\defaultcohDegree}{q}

\newcommand{\defaultlclocus}{S}

\newcommand{\defaultvphi}{\vphi_L}
\newcommand{\setDefaultvphi}[1]{\renewcommand{\defaultvphi}{#1}}

\newcommand{\defaultpsi}{\psi}

\newcommand{\defaultMetric}{\omega}
\newcommand{\setDefaultMetric}[1]{\renewcommand{\defaultMetric}{#1}}

\newcommand{\defaultMap}{\pi}

\DeclareMathAlphabet{\mathbfkurier}{OT1}{kurier}{b}{it}

\newcommand{\alert}[2][RoyalBlue]{{\color{#1}#2}}

\NewDocumentCommand{\logKX}{
  D<>{\defaultAmbientSpace}  
  t{.}                       
  o                          
}{K_{#1} \IfBooleanT{#2}{\otimes \defaultlclocus} \IfNoValueF{#3}{\otimes #3}}

\NewDocumentCommand{\Ltwo}{ 
  t{.}                      
  D//{\bullet,\bullet}      
  D<>{\defaultAmbientSpace} 
  s                         
  m                         
}{L^{#2}_{(2)\IfBooleanT{#1}{\:\text{loc}}}\paren{\IfBooleanF{#4}{#3;} #5}}

\newcommand{\HarmSym}{\mathcal{H}}
\NewDocumentCommand{\Harm}{ 
  t{'}                      
  D//{\defaultcohDegree}    
  D<>{\defaultAmbientSpace} 
  g                         
  t{,}                      
  G{\defaultvphi}           
  e{_}                      
}{\HarmSym^{\IfBooleanT{#1}{\defaultDimension,}#2}\IfNoValueF{#4}{\paren{#3;#4}}_{#6 \IfNoValueF{#7}{,#7}}}

\newcommand{\deltaH}{\delta_{\HarmSym}}

\NewDocumentCommand{\lcIndex}{ 
  m  
  m  
  m  
}{#1\IfNoValueF{#2}{+#2}\IfNoValueF{#3}{-#3}}

\NewDocumentCommand{\lcData}{ 
  G{\defaultvphi}  
  O{\defaultpsi}   
  e{.}             
}{\paren{#1; \IfNoValueF{#3}{#3 \cdot} #2}}

\NewDocumentCommand{\lcdata}{ 
  s                
  d<>              
  G{\defaultvphi}  
  O{\defaultpsi}   
  e{.,}            
}{{\newcommand{\datalist}{\IfNoValueF{#2}{#2,}#3,#4\IfNoValueF{#5}{,#5}\IfNoValueF{#6}{,#6}}
\IfBooleanTF{#1}{\datalist}{\paren{\datalist}}}}

\NewDocumentCommand{\drR}{ 
  m              
  O{\defaultMap} 
  g              
  s              
}{R^{#1}{#2}_*\IfNoValueF{#3}{\IfBooleanTF{#4}{#3}{\paren{#3}}}}

\newcommand{\spHsym}{\mathbb{H}}
\NewDocumentCommand{\spH}{ 
  O{\spHsym}              
  D//{\defaultcohDegree}  
  t{M}                    
  m                       
}{#1^{#2}\paren{\IfBooleanT{#3}{M\otimes}#4}}

\def\spHarm{\spH[\HarmSym]}

\DeclareMathOperator{\lc}{lc} 
\NewDocumentCommand{\lcc}{ 
  D||{\defaultlcIndex}       
  e{+-}                      
  D<>{\IfBooleanF{#7}{\defaultAmbientSpace}}  
  t{'}                       
  D(){\defaultlclocus}       
  s                          
}{\lc_{#4}^{\lcIndex{#1}{#2}{#3}}\IfBooleanF{#7}{\IfBooleanTF{#5}{\paren{#6}}{\lcData}}}

\NewDocumentCommand{\lcS}{  
  s                       
  D(){\defaultlclocus}    
  D||{\defaultlcIndex}    
  e{+-}                   
  d<>                     
  O{p}                    
  t{'}                    
}{\IfBooleanT{#8}{\mathring}{\mathtt{\IfBooleanT{#1}{\rs} #2}}^{\lcIndex{#3}{#4}{#5}}_{\IfNoValueF{#6}{#6,}#7}}

\NewDocumentCommand{\lcZ}{o}{\mathtt Z\IfNoValueF{#1}{^{#1}}}

\NewDocumentCommand{\PRes}{ 
  O{}      
  d()      
}{\mathcal R_{#1}\IfNoValueF{#2}{\paren{#2}}}

\NewDocumentCommand{\HRes}{ 
  d()   
}{\mathfrak{R}\IfNoValueF{#1}{\paren{#1}}}

\newcommand{\defidlof}[1]{\mathcal{I}_{#1}}  
\NewDocumentCommand{\mtidlof}{   
  O{}      
  D<>{#1}  
  m        
}{\multidl_{#2}\paren{#3}} 

\NewDocumentCommand{\residlof}{  
  D||{\defaultlcIndex}   
  e{+-}                  
  d<>                    
  s                      
}{\sheaf R_{\IfNoValueTF{#4}{}{#4,} \lcIndex{#1}{#2}{#3}}\IfBooleanF{#5}{\lcData}}

\NewDocumentCommand{\Adjidlof}{
  D||{\defaultlcIndex}       
  D<>{\defaultAmbientSpace}  
  D(){\defaultlclocus}       
  m                          
}{\operatorname{\mathit{Adj}}^{#1}_{\paren{#2,#3}}\paren{#4}}

\NewDocumentCommand{\aidlof}{
  D||{\defaultlcIndex}   
  e{+-}                  
  d<>                    
  s                      
  e{',;}                 
}{\sheaf{J}_{\!\IfNoValueF{#4}{#4,} \lcIndex{#1}{#2}{#3}
    \IfNoValueF{#7}{, #7}\IfNoValueF{#8}{; #8}} \IfNoValueF{#6}{^{#6}} \IfBooleanF{#5}{\lcData}}

\NewDocumentCommand{\faidlof}{
  D||{\defaultlcIndex}   
  e{+-}                  
  t{/}                   
  D||{\defaultlcIndex}   
  e{+-}                  
}{\fracAidlof{\lcIndex{#1}{#2}{#3}}{\lcIndex{#5}{#6}{#7}}}

\NewDocumentCommand{\fracAidlof}{
  m                  
  m                  
  d<>                
  s                  
  e{',}              
  G{\defaultvphi}    
  O{\defaultpsi}     
  e{.}               
}{\frac{
    \aidlof|#1|<#3>* 
    \IfBooleanF{#4}{\lcData{#7}[#8].{#9}}
  }{
    \aidlof|#2|<#3>* 
    \IfBooleanF{#4}{\lcData{#7}[#8].{#9}}
  }}

\RenewDocumentCommand{\Res}{
  D||{\defaultlcIndex} 
  e{+-}                
  e{^}                 
}{\operatorname{Res}^{\IfNoValueTF{#4}{\lcIndex{#1}{#2}{#3}}{#4}}}

\NewDocumentCommand{\lcV}{ 
  D||{\defaultlcIndex}    
  D//{\defaultvphi}       
  d()                     
  e{^}                    
  O{\defaultpsi}          
}{\:d\operatorname{lcv}^{#1\IfNoValueF{#4}{,\paren{#4}}}_{\IfNoValueF{#3}{#3,}#2}\left[#5\right]}

\NewDocumentCommand{\Ohvol}{ 
  D//{\defaultvphi} 
  d()               
  O{\defaultpsi}    
}{\dvol_{\IfNoValueF{#2}{#2,}#1}\left[#3\right]}

\newcommand{\dvol}{\:d\vol}

\newcommand{\RTFsym}{\mathfrak{F}} 
\NewDocumentCommand{\RTF}{ 
  s          
  G{\RTFsym} 
  o          
  >{\SplitArgument{1}{,}} d<> 
  d||        
  D(){\eps}  
  t{,}       
}{%
  \begingroup%
    \newif\ifsmasht%
    \IfBooleanTF{#1}{\smashttrue}{\smashtfalse}%
    \newif\ifboolup%
    \booluptrue%
    \IfNoValueT{#3}{\IfNoValueT{#4}{\IfNoValueT{#5}{\boolupfalse}}}%
    \newcommand{\supsrptstr}{\IfNoValueF{#3}{#3}\IfNoValueF{#4}{\inner#4}\IfNoValueF{#5}{\abs{#5}^2}}
    \newcommand{\RTFvar}{#6}
    #2\RTFprocess
}

\NewDocumentCommand{\RTFprocess}{
  o                     
  d<>                   
  t{,}                  
  G{}                   
  D||{\defaultlcIndex}  
  e{+-}                 
  t{.}                  
}{\newcommand{\subsrptstr}{%
    \IfNoValueTF{#1}{
    \IfNoValueF{#2}{#2,}
    \IfBooleanT{#3}{\lcdata*#4\relax,}
    \lcIndex{#5}{#6}{#7}}{#1}}%
  \newcommand{\srptstr}{\cramped{{}^{\supsrptstr}%
      \ifboolup _
      \fi{\ifboolup\displaystyle\fi\paren{\RTFvar}%
          \ifboolup {\scriptstyle \subsrptstr} \else _{\subsrptstr} \fi%
        }}}%
  \ifboolup%
    \ifsmasht%
      \smash[t]{
        \raisebox{\depthof{$\srptstr$} * \real{0.3}}{$\srptstr$}%
      }%
    \else%
      \raisebox{\depthof{$\srptstr$} * \real{0.3}}{$\srptstr$}%
    \fi%
  \else%
    \srptstr%
  \fi%
  \endgroup%
}

\NewDocumentCommand{\mtlog}{O{e} d() D||{\defaultpsi}}{\log\!#1^{\paren{#2}}\abs{#3}}
\NewDocumentCommand{\slog}{O{e} D||{\defaultpsi}}{\log\abs{#1 #2}}
\NewDocumentCommand{\dlog}{O{e} D||{\defaultpsi}}{\mtlog[#1](2)|#2|}

\NewDocumentCommand{\logpole}{ 
  D||{\defaultpsi}       
  t{,}                   
  D||{\defaultlcIndex}   
  e{+-}                  
  E{.^}{{e}{1+\eps}}     
  s                      
}{\abs{#1}^{\lcIndex{#3}{#4}{#5}} \IfBooleanTF{#8}{\slog[#6]|#1|}{\paren{\slog[#6]|#1|}^{#7}}}

\DeclareFontFamily{OMX}{MnSymbolE}{}
\DeclareSymbolFont{MnLargeSymbols}{OMX}{MnSymbolE}{m}{n}
\SetSymbolFont{MnLargeSymbols}{bold}{OMX}{MnSymbolE}{b}{n}
\DeclareFontShape{OMX}{MnSymbolE}{m}{n}{
    <-6>  MnSymbolE5
   <6-7>  MnSymbolE6
   <7-8>  MnSymbolE7
   <8-9>  MnSymbolE8
   <9-10> MnSymbolE9
  <10-12> MnSymbolE10
  <12->   MnSymbolE12
}{}
\DeclareFontShape{OMX}{MnSymbolE}{b}{n}{
    <-6>  MnSymbolE-Bold5
   <6-7>  MnSymbolE-Bold6
   <7-8>  MnSymbolE-Bold7
   <8-9>  MnSymbolE-Bold8
   <9-10> MnSymbolE-Bold9
  <10-12> MnSymbolE-Bold10
  <12->   MnSymbolE-Bold12
}{}
\DeclareMathDelimiter{\llangle}{\mathopen}%
{MnLargeSymbols}{'164}{MnLargeSymbols}{'164}
\DeclareMathDelimiter{\rrangle}{\mathclose}%
{MnLargeSymbols}{'171}{MnLargeSymbols}{'171}

\newcommand{\iinner}[2]{\left\llangle#1,#2\right\rrangle}

\NewDocumentCommand{\idxup}{ 
  m                  
  O{\defaultMetric}  
  t{,}               
  o                  
  s                  
  t{.}               
}{\paren{#1}^{
    \!\IfBooleanTF{#5}{\smash[t]{#2}}{#2}\IfNoValueF{#4}{, #4}
  }\IfBooleanT{#6}{\!\!\ctrt}}

\newcommand{\dbadj}{\dbar^{\smash{\mathrlap{*}\;\:}}}

\NewDocumentCommand{\sm}{s m}{{#2}\IfBooleanTF{#1}{_}{^}\text{sm}}

\NewDocumentCommand{\idx}{ 
  O{i} 
  m    
  o    
  t{.} 
  t{,} 
  o    
  m    
}{{#1}_{#2} \IfNoValueF{#3}{#3}
  \IfBooleanT{#4}{\dotsm} \IfBooleanT{#5}{,\dots,}
  \IfNoValueF{#6}{#6} {#1}_{#7}}

\newcommand{\defaultPscript}{p}

\NewDocumentCommand{\ps}{ 
  O{\nu_{\IfNoValueTF{#4}{#3}{#4}}} 
  t{.}                              
  E{_^}{{\defaultPscript}}          
}{\IfNoValueTF{#4}{_{\IfBooleanT{#2}{#1 \cdot}(#3)}}{^{\IfBooleanT{#2}{#1 \cdot}(#4)}}}

\DeclareFontFamily{U} {MnSymbolC}{}
\DeclareSymbolFont{MnSyC} {U} {MnSymbolC}{m}{n}
\SetSymbolFont{MnSyC} {bold}{U} {MnSymbolC}{b}{n}
\DeclareFontShape{U}{MnSymbolC}{m}{n}{
  <-6> MnSymbolC5
  <6-7> MnSymbolC6
  <7-8> MnSymbolC7
  <8-9> MnSymbolC8
  <9-10> MnSymbolC9
  <10-12> MnSymbolC10
  <12-> MnSymbolC12}{}
\DeclareFontShape{U}{MnSymbolC}{b}{n}{
  <-6> MnSymbolC-Bold5
  <6-7> MnSymbolC-Bold6
  <7-8> MnSymbolC-Bold7
  <8-9> MnSymbolC-Bold8
  <9-10> MnSymbolC-Bold9
  <10-12> MnSymbolC-Bold10
  <12-> MnSymbolC-Bold12}{}
\DeclareMathSymbol{\smallstar}{\mathbin}{MnSyC}{"80}
\DeclareMathSymbol{\medstar}{\mathbin}{MnSyC}{"82}
\DeclareMathSymbol{\largestar}{\mathbin}{MnSyC}{"83}

\newcommand{\charfct}{\mathbbm 1}
\NewDocumentCommand{\lelong}{m O{x}}{\operatorname{\boldsymbol{\nu}}\paren{#1,#2}}

\newcommand{\cvr}[1]{\mathfrak{#1}} 
\NewDocumentCommand{\rs}{ 
  s  
  m  
}{\IfBooleanTF{#1}{\smash[t]{\widetilde{#2}}}{\widetilde{#2}}}

\NewDocumentCommand{\wh}{ 
  s  
  m  
}{\IfBooleanTF{#1}{\smash[t]{\widehat{#2}}}{\widehat{#2}}}

\DeclareMathOperator{\Ann}{Ann}  
\DeclareMathOperator{\mlc}{mlc} 
\newcommand{\Diff}{\operatorname{Diff}^*} 

\newcommand{\sect}[1][s]{\mathtt{#1}} 
\newcommand{\bphi}{\boldsymbol{\vphi}}

\NewDocumentCommand{\cbn}{  
  D//{\defaultlcIndex_V}
  D||{\defaultlcIndex}
}{\mathfrak{C}^{#1}_{#2}} 
\NewDocumentCommand{\Iset}{  
  D||{\defaultlcIndex}    
  e{+-}                   
  O{\defaultlclocus}      
}{I^{\lcIndex{#1}{#2}{#3}}_{#4} 
} 

\ifcsname defineNoThmInMarionotations\endcsname
  \relax
\else 

  \newtheorem{THMprop}{Proposition}[subsection]
  \newtheorem{THMlemma}[THMprop]{Lemma}
  \newtheorem{THMthm}[THMprop]{Theorem}
  \newtheorem{THMcor}[THMprop]{Corollary}

  \newtheorem{THMconjecture}[THMprop]{Conjecture}
  \newtheorem*{THMclaim}{Claim}

  \def\makeparenother{\catcode`\(=12 \catcode`\)=12 }
  \def\makeparenactive{\catcode`\(=\active\catcode`\)=\active}

  \makeparenactive
  \NewDocumentEnvironment{textupparenenvir}{}{

    \everymath\expandafter{\makeparenother}
    \everydisplay\expandafter{\makeparenother}

    \def({\textup{\char`\(}}
    \def){\textup{\char`\)}}

    \makeparenactive
  }{\makeparenother}
  \makeparenother

  \NewDocumentEnvironment{prop}{ +o }{
    \IfNoValueTF{#1}{\begin{THMprop}}{\begin{THMprop}[{#1}]}
      \begin{textupparenenvir}
  }{
      \end{textupparenenvir}
    \end{THMprop}
  }

  \NewDocumentEnvironment{lemma}{ +o }{
    \IfNoValueTF{#1}{\begin{THMlemma}}{\begin{THMlemma}[{#1}]}
      \begin{textupparenenvir}
  }{
      \end{textupparenenvir}
    \end{THMlemma}
  }

  \NewDocumentEnvironment{thm}{ +o }{
    \IfNoValueTF{#1}{\begin{THMthm}}{\begin{THMthm}[{#1}]}
      \begin{textupparenenvir}
  }{
      \end{textupparenenvir}
    \end{THMthm}
  }

  \NewDocumentEnvironment{cor}{ +o }{
    \IfNoValueTF{#1}{\begin{THMcor}}{\begin{THMcor}[{#1}]}
      \begin{textupparenenvir}
  }{
      \end{textupparenenvir}
    \end{THMcor}
  }

  \NewDocumentEnvironment{conjecture}{ +o }{
    \IfNoValueTF{#1}{\begin{THMconjecture}}{\begin{THMconjecture}[{#1}]}
      \begin{textupparenenvir}
  }{
      \end{textupparenenvir}
    \end{THMconjecture}
  }

  \NewDocumentEnvironment{claim}{ +o }{
    \IfNoValueTF{#1}{\begin{THMclaim}}{\begin{THMclaim}[{#1}]}
      \begin{textupparenenvir}
  }{
      \end{textupparenenvir}
    \end{THMclaim}
  }

  \theoremstyle{remark}
  \newtheorem{remark}[THMprop]{Remark}

  \theoremstyle{definition}

  \newtheorem{notation}[THMprop]{Notation}

  \numberwithin{equation}{subsection}

  \patchcmd{\subsubsection}{\normalfont}{\normalfont\color{blue}}{}{}

\fi

\allowdisplaybreaks  

\begin{document}

\newcommand{\titlestr}{%
  An extension theorem in terms of adjoint ideal sheaves%
}

\newcommand{\shorttitlestr}{%
  An extension theorem in terms of adjoint ideal sheaves%
}

\newcommand{\MCname}{Tsz On Mario Chan}
\newcommand{\MCnameshort}{Mario Chan}
\newcommand{\MCemail}{mariochan@ntu.edu.tw}

\newcommand{\addressstr}{%
  Dept.~of Mathematics,
  National Taiwan University, Taiwan
}

\newcommand{\subjclassstr}[1][,]{%
  32J25 (primary)#1  
  32Q15#1            
  14B05 (secondary)
}

\newcommand{\keywordstr}[1][,]{%
  adjoint ideal sheaf#1
  multiplier ideal sheaf#1
  lc centre%
}

\newcommand{\dedicatorystr}{%
}

\newcommand{\thankstr}{%
  This work was partly supported by the National Research Foundation (NRF) of
  Korea grant funded by the Korea government (No.~2023R1A2C1007227)%
}


\title[\shorttitlestr]{\titlestr}
 
\author[\MCnameshort]{\MCname}
\email{\MCemail}
\address{\addressstr}


\thanks{\thankstr}
 
\subjclass[2020]{\subjclassstr}

\keywords{\keywordstr}


\begin{abstract}

In the context of the study of the Ohsawa--Takegoshi $L^2$ extension
theorem, a ``qualitative'' extension theorem in terms of adjoint ideal
sheaves (i.e.~extension result without $L^2$ estimates but with some
form of local $L^2$ condition ensured) is proved on compact Kähler
manifolds via harmonic theory and residue computations.
The arguments are adapted from those in the study of the injectivity
theorem by Chan--Choi--Matsumura.
The result guarantees, in particular, the existence of extensions of
line-bundle-valued holomorphic top forms over the union of
any log-canonical centres of the same codimension to top forms over
the ambient space under suitable positivity assumptions. 


\end{abstract} 

\date{February 16, 2025 (last updated: \today)}

\maketitle



\subsection{Introduction}
\label{sec:intro}













The goal of this note is to improve the ``qualitative'' extension
result in terms of adjoint ideal sheaves given in
\cite{Chan&Choi_injectivity-proceedings}*{Sec.~3}, which states that,
for certain holomorphic line-bundle-valued top forms on some union of
log-canonical (lc) centres (possibly of high codimension), under
suitable (weak) positivity assumption on the line bundle, they can be
holomorphically extended to some line-bundle-valued top forms over the
ambient compact Kähler manifold which take values in some adjoint
ideal sheaves (i.e.~some local $L^2$ condition), even though the
extensions come without $L^2$ estimates.
As in the statement in \cite{Chan&Choi_injectivity-proceedings}*{Sec.~3},
the proof of the extension result here follows the strategy of the
(second) proof of the extension result given in 
the work of Cao, Demailly and Matsumura
\cite{Cao&Demailly&Matsumura}*{Sec.~3}, which makes use of the
arguments in the proof of the injectivity theorem (see
\cites{Matsumura_injectivity-lc, Chan&Choi_injectivity-I,
  Chan&Choi&Matsumura_injectivity}).

Such result can be considered as a variant of the celebrated
Ohsawa--Takegoshi $L^2$ extension theorem \cite{Ohsawa&Takegoshi-I},
which is of fundamental importance in several complex variables and
complex geometry, as reflected from the enormous amount of researches
(see \citelist{
  \cite{Berndtsson_OT-extension}
  \cite{Berndtsson_ext-form}
  \cite{Berndtsson&Lempert}
  \cite{Blocki_Suita-conj}
  \cite{Cao&Paun_OT-ext}
  \cite{Cao&Demailly&Matsumura}
  \cite{Chan_on-L2-ext-with-lc-measures}
  \cite{Chan&Choi_ext-with-lcv-codim-1}
  \cite{Demailly_on_OTM-extension}
  \cite{Demailly_extension}
  \cite{DHP}
  \cite{Guan&Zhou_optimal-L2-estimate}
  \cite{Guan&Zhou_effective_openness}
  \cite{Guan&Zhou&Zhu_OTExt}
  \cite{KimDano_lc-extension}
  \cite{KimDano-L2-ext-for-lc}
  \cite{Manivel}
  \cite{McNeal&Varolin_adjunction}
  \cite{McNeal&Varolin_L2-estimates}
  \cite{Ohsawa-II}
  \cite{Ohsawa-III}
  \cite{Ohsawa&Takegoshi-I}
  \cite{Paun_inv-plurigenera}
  \cite{Siu_inv_plurigenera2}
}, which is by no means exhaustive) devoted to its generalisations in
different directions in order to fit in various applications.

Adjoint ideal sheaf arises from the comparison of
$\res{\mtidlof<X>{\vphi}}_S$, the restriction of a multiplier ideal
sheaf of some plurisubharmonic (psh) function $\vphi$ on a complex
manifold $X$ to a subvariety $S$, with $\mtidlof<S>{\vphi}$, the
multiplier ideal sheaf of the restricted data $\res\vphi_S$ on $S$
(see \cite{Lazarsfeld_book-II}*{\S 9.3.E} and
\cite{Takayama_adj-ideal}*{Prop.~2.4}, or
\cite{Chan_adjoint-ideal-nas}*{Sec.~1.1}).
The comparison is described via the \emph{residue exact sequence}
\eqref{eq:residue-seq}. 
It is therefore natural to consider adjoint ideal sheaves when one
tries to extend holomorphic sections from a subspace to an ambient
space with $L^2$ conditions.
It turns out that adjoint ideal sheaves provide a sheaf-theoretic
means to describe the finer structure of the non-integrable locus
given by a multiplier ideal sheaf (as seen from the filtration
\eqref{eq:aidl_filtration} and its description of lc centres in
Section \ref{sec:notation-in-snc}).
While there are already versions of the Ohsawa--Takegoshi extension
theorem that can extend holomorphic sections over a high-codimensional
subvariety \citelist{\cite{Manivel}\cite{Demailly_on_OTM-extension}},
or an lc centre \cite{KimDano_lc-extension}, or a divisor with simple
normal crossings (snc)
\citelist{\cite{Demailly_extension}\cite{Cao&Paun_OT-ext}} (with
varying positivity and regularity assumptions on the relevant
metrics) and most of them come with $L^2$ estimates on the
extensions, none of the estimates provide a direct control of which
adjoint ideal sheaf that an extension is situated in.
Furthermore, the description of an adjoint ideal sheaf is, to an
extent, invariant under the log-resolution of the given data (see
\cite{Chan_adjoint-ideal-nas}*{Sec.~5}), making it a suitable language
for the use in birational/bimeromorphic geometry.
The use of them also makes it easy to generalise some statements on
smooth spaces to the ones on certain singular spaces, as illustrated in
\cites{Chan&Choi&Matsumura_injectivity,
  Chan&Choi&Matsumura_injectivity-II} where the injectivity theorems
on compact/holomorphically convex Kähler manifolds are generalised to
globally embedded snc spaces.
With all these backgrounds, the author believes that it is worth
incorporating adjoint ideal sheaves into the study of extension
problem, even though the classical techniques used in the
Ohsawa--Takegoshi extension may not be directly applicable, and an
appropriate $L^2$ estimate in general may still be far from achievable
for the moment.

Let $X$ be a compact Kähler manifold of dimension $n$ and $L$ be a
holomorphic line bundle over $X$ equipped with a singular hermitian metric
$e^{-\vphi_L}$ such that the (local) potential $\vphi_L$ has neat analytic
singularities (such that there exists a log-resolution $\pi \colon \rs
X \to X$ of $(X,\vphi_L)$ so that the polar divisor of $\pi^*\vphi_L$
is an $\fieldR$-divisor with snc, and the multiplier ideal sheaf
$\mtidlof{\vphi_L}$ of $\vphi_L$ is coherent).
Also let $\psi$ be a quasi-plurisubharmonic (quasi-psh) global
function on $X$ which satisfies $\psi \leq -1$ and has neat analytic
singularities.
Suppose that $\psi$ is not smooth.
By the strong openness property of multiplier ideal sheaves of psh
functions (which is proved by Guan and Zhou \cite{Guan&Zhou_openness}
when both $\vphi_L$ and $\psi$ are quasi-psh and have arbitrary
singularities; see also \cite{Guan&Zhou_effective_openness} and
\cite{Hiep_openness}), there exist numbers $m_{k-1}$ and $m_k$ such
that $m_k > m_{k-1} \geq 0$ and 
\begin{equation*}
  \mtidlof{\vphi_L+m_k\psi} \subsetneq \mtidlof{\vphi +m \psi}
  = \mtidlof{\vphi_L +m_{k-1} \psi}
  \quad\text{ on } X
  \quad\text{ for all } m \in [m_{k-1}, m_k) \; , 
\end{equation*}
i.e.~$m_k$ is a \emph{jumping number} of the system $\lcdata<X>$, or
more precisely, of the family $\set{\mtidlof{\vphi_L +m\psi}}_{m \geq
  0}$.
The adjoint ideal sheaves $\aidlof{\vphi_L}.{m_k} :=
\aidlof<X>{\vphi_L}.{m_k}$ (of index $\sigma$) under consideration is
the one defined in \cite{Chan_adjoint-ideal-nas}, given at each $x \in
X$ by
\begin{equation*}
  \aidlof{\vphi_L}.{m_k}_x
  :=\setd{f\in \holo_{X,x}}{
    \begin{multlined}[c][0.45\textwidth]
      \exists~\text{open set } V_x \ni x \: , \; \forall~\eps > 0 \: ,
      \\
      \eps \int_{\mathrlap{V_x}}\; \frac{\abs f^2 e^{-\vphi_L -m_k
          \psi} \dvol_{V_x}}{\logpole} < +\infty
    \end{multlined}
  } \; .
\end{equation*}
By \cite{Chan_adjoint-ideal-nas}*{Thm.~1.2.3 (1)}, there is a
filtration
\begin{equation} \label{eq:aidl_filtration}
  \xymatrix@C-1.6em@R=1em{
    {\aidlof|0|{\vphi_L}.{m_k}}
    \ar@{}[r]|-*+{\subset}
    & {\aidlof|1|{\vphi_L}.{m_k}}
    \ar@{}[r]|-*+{\subset}
    & {\dotsm}
    \ar@{}[r]|-*+{\subset}
    & {\aidlof|\sigma_{\mlc}|{\vphi_L}.{m_k}}
    \ar@{}[r]|-*+{=}
    & {\aidlof|\sigma_{\mlc}|+1*(\dotsm)}
    \ar@{}[r]|-*+{=}
    & {\dotsm}
    \\
    {\mtidlof{\vphi_L+m_k \psi}}
    \ar@{}[u]|(0.4)*[left]+{=}
    &&& {\mtidlof{\vphi_L +m_{k-1} \psi}}
    \ar@{}[u]|(0.4)*[left]+{=}
  }
\end{equation}
for some integer $\sigma_{\mlc} \in \set{1, \dots, n}$.
(Note that the ad hoc ideal sheaf $\mathscr I'\paren{m_{k-1}\psi}$
introduced in \cite{Demailly_extension}, i.e.~the subsheaf of
$\mtidlof{\vphi_L +m_{k-1} \psi}$ consisting of all the germs locally
$L^2$ with respect to the Ohsawa measure induced from
$\vphi_L+m_k\psi$, corresponds to $\aidlof|1|{\vphi_L}.{m_k}$ by
\cite{Chan_Residue-fct-proceedings}*{Cor.~2.2.2}.)

After a ``normalisation'' (see \cite{Chan_adjoint-ideal-nas}*{footnote
  9 in Sec.~2.2}), assume that $m_k = 1$ and $m_{k-1} = 0$ and we
focus on the system $\lcdata<X>$ at the jumping number $m
= 1$.
Also write $\aidlof := \aidlof.1$ for convenience.

For the purpose of this note, assume that $\lcdata<X>$ is in an \emph{snc
configuration}, i.e.~the polar varieties of $\vphi_L$ and $\psi$ are
snc divisors (in particular, the polar loci $P_{\vphi_L} :=
\vphi_L^{-1}(-\infty)$ and $P_{\psi} :=\psi^{-1}(-\infty)$ are snc
divisors) and the sum $P_{\vphi_L} + P_{\psi}$ is also snc.

In the ``qualitative'' extension result in terms of adjoint ideal sheaves
given in \cite{Chan&Choi_injectivity-proceedings}*{\S 3}, 
the polar divisor of the quasi-psh function $\psi$ is already a
reduced divisor which is the lc locus $S$ of the system $\lcdata<X>$
(i.e.~the reduced variety defined by the radical annihilator
$\Ann_{\holo_X} \faidlof|\sigma_{\mlc}|/|0|*$).
In this note, the extension result is generalised to allow a more general
quasi-psh $\psi$ whose polar divisor may not be reduced and may
contain components outside of the lc locus.
Moreover, the proof of the statement is rewritten to make use of the
notion of ``harmonic residue'' developed in the study of the
injectivity theorem in \cite{Chan&Choi&Matsumura_injectivity} and
\cite{Chan&Choi&Matsumura_injectivity-II} (so that all the arguments
are made via the use of equalities instead of inequalities).
It is hoped that this will give more insight in generalising the
result (see Remark \ref{rem:weaker-assumptions}) so that it is
applicable in practical situations. 

The statement is given as follows.
\begin{thm}[Theorem \ref{thm:global-extension}]
  Write $\aidlof* := \aidlof$ for all integers $\sigma \geq 0$ and set
  $\aidlof|-1|* := 0$.
  Also write
  \begin{equation*}
    \spH{\sheaf F} :=\cohgp q[X]{K_X\otimes L \otimes \sheaf F}
    \quad\text{ for any sheaf } \sheaf F \text{ on } X
  \end{equation*}
  for convenience.
  Suppose that \eqref{eq:curvature-assumption} holds true for a given
  integer $\sigma \in [0, \sigma_{\mlc}]$, i.e.~the
  curvature assumptions in terms of functions derived from
  $\vphi_L$ and $\psi$ are satisfied on each $\sigma$-lc centre
  (i.e.~an irreducible component of the reduced variety $\lcc'$
  defined by $\Ann_{\holo_X} \faidlof/-1*$, which is purely
  $\sigma$-codimensional in $X$ when $\lcdata<X>$ is in an snc
  configuration).
  Then, the short exact sequence
  \begin{equation*}
    \xymatrix@1{
      *+<1ex>{0} \ar[r]
      &*+<1ex>{\faidlof/-1*} \ar[r]
      &*+<1ex>{\faidlof|\sigma'|/-1*} \ar[r]
      &*+<1ex>{\faidlof|\sigma'|*} \ar[r]
      &*+<1ex>{0}
    }
  \end{equation*}
  for any $\sigma' \geq \sigma$ induces a long exact sequence which
  splits into short exact sequences 
  \begin{equation*}
    \xymatrix{
      {0} \ar[r]
      &{\spH{\faidlof/-1*}} \ar[r]^-{\tau_\sigma^{\sigma'}}
      &{\spH{\faidlof|\sigma'|/-1*}} \ar[r]
      &{\spH{\faidlof|\sigma'|/*}} \ar[r]
      &{0}
    } \quad\text{ for all } q \geq 0 \; .
  \end{equation*}
  In particular, for every $\sigma' > \sigma$, each $f \in
  \spH/0/{\faidlof|\sigma'|*}$ (a section over $\lcc+1'$) has an
  ``extension'' $F_{\sigma -1} \in \spH/0/{\faidlof|\sigma'|/-1*}$ (a
  section over $\lcc'$) such that $f \equiv F_{\sigma -1} \mod
  \aidlof*$.
  (When $\sigma' = \sigma_{\mlc}$ and $\sigma = 0$, $f$ is a
  section over $S$ and $F_{-1}$ is a section over $X$, which is the
  extension given in \cite{Cao&Demailly&Matsumura}*{Sec.~3}.)

  Furthermore, if \eqref{eq:curvature-assumption} holds for every
  $\sigma \geq 0$, it follows by induction that, for every $\sigma' \geq 0$, every $f \in
  \spH/0/{\faidlof|\sigma'|/|\sigma'-1|*}$ (a section over $\lcc|\sigma'|'$) 
  has an extension $F_{-1} \in \spH/0/{\faidlof|\sigma'|/|-1|*}
  =\spH/0/{\aidlof|\sigma'|*}$ (a section over $X$) such that $f \equiv
  F_{-1} \mod \aidlof|\sigma' -1|*$.
\end{thm}

As can be seen from the above statement, the extension is actually
done inductively, $1$ dimension at a time.
It is then natural to ask whether the known $L^2$ estimates for
extensions from subvarieties of codimension $1$ can be used
inductively to provide an estimate for $F_{-1}$.

Also see Remarks \ref{rem:curv-assumption-on-X_only} and
\ref{rem:weaker-assumptions} for some speculations on weakening the
positivity assumption \eqref{eq:curvature-assumption}.

The notation and convention used in this note follow mostly those in
\citelist{
  \cite{Chan_adjoint-ideal-nas}
  \cite{Chan_Residue-fct-proceedings}
  \cite{Chan&Choi_injectivity-I}
  \cite{Chan&Choi&Matsumura_injectivity}
  \cite{Chan&Choi&Matsumura_injectivity-II}}.
The following convention is recalled here for clarity.
\begin{notation} \label{notation:potentials} \
  \begin{itemize}[align=left, itemindent=*, leftmargin=*]
  \item A \emph{potential} $\vphi_L$ (of the curvature of a metric
    $e^{-\vphi_L}$) on a holomorphic line bundle $L$ is considered as
    a collection of local functions $\set{\vphi_\gamma}_\gamma$ (with
    respect to an open cover of $X$) whose pairwise differences
    $\vphi_\gamma - \vphi_{\gamma'}$ are pluriharmonic;
  
  \item For any prime (Cartier) divisor $E$ on $X$,
    \begin{itemize}[label=$\triangleright$]
    \item ``$E$'' is abused to mean also its associated line bundle,
    \item $\phi_E$ denotes the potential $\log\abs{\sect_E}^2$ on $E$
      induced from a canonical section $\sect_E$ of $E$,
    \item $\sm\vphi_E$ denotes a smooth potential on $E$,
    \item $\psi_E := \phi_E -\sm\vphi_E$ is a global function on $X$,
      normalised by adding a suitable constant such that
      $\psi_E \leq -1$.
    \end{itemize}
    The above notations are extended to any $\fieldR$-divisor $E$ by
    linearity.
  \end{itemize}
\end{notation}


\subsection{Adjoint ideal sheaves, lc centres and residue exact sequences}
\label{sec:notation-in-snc}

The notions associated with adjoint ideal sheaves are recalled from
\cite{Chan_adjoint-ideal-nas} in this section for the ease of
reference as well as fixing notation.
Keep writing $\aidlof* := \aidlof$ for all $\sigma \geq 0$ and
$\aidlof|-1|* := 0$ in what follows.

Throughout this paper, the system $\lcdata<X>$ is assumed to be
in the snc configuration.
The lc locus $S$, defined by the radical ideal sheaf $\defidlof{S} :=
\Ann_{\holo_X} \faidlof|\sigma_{\mlc}|/|0|*$ (thus $S \subset
P_{\psi}$), is a reduced snc divisor.
Moreover, each $\sigma$-lc centre $\lcS$ in the union $\lcc' :=
\bigcup_{p \in \Iset} \lcS$ (defined by the radical ideal sheaf
$\defidlof{\lcc'} :=\Ann_{\holo_X} \faidlof/-1*$) is an irreducible
component (of codimension $\sigma$ in $X$) of intersection of
irreducible components of $S$.
They coincide with the $\sigma$-codimensional lc centres of the
log-canonical (lc) pair $(X,S)$ in the minimal model program (see, for
example, \cite{Kollar_Sing-of-MMP}*{Def.~4.15}).

Recall from \cite{Chan_adjoint-ideal-nas}*{Sec.~2.3} that, under the
snc assumption, there is a decomposition
\begin{equation*}
  \vphi_L+\psi =\bphi +\phi_{S_0} +\phi_S
\end{equation*}
such that the polar locus $P_{\bphi} := \bphi^{-1}(-\infty)$ is a
reduced snc divisor which contains no irreducible components of $S$,
and $S_0$ is a divisor with $\supp S_0 \leq S$.
Then $\bphi$ is a potential on
\begin{equation*}
  L' := L \otimes S_0^{-1} \otimes S^{-1} \; .
\end{equation*}
Unless stated otherwise, it is assumed in the rest of the paper that
\begin{equation*}
  \ibddbar\paren{\vphi_L + \psi} \geq 0 \; ,
  \quad\text{ or equivalently, }\quad
  \ibddbar\bphi \geq 0 \quad\text{ on } X \; .
\end{equation*}

According to \cite{Chan_adjoint-ideal-nas}*{Thm.~4.1.2 (1)}, under the
snc assumption, the adjoint ideal sheaves decompose into
\begin{equation*}
  \aidlof = \mtidlof{\bphi+\phi_{S_0}} \cdot \defidlof{\lcc+1'}
  =\mtidlof{\vphi_L
    \alert[Gray]{{} +{}\smash{\cancelto{0}{m_{k-1} \vphantom{m^m}}} \psi}} \cdot
  \defidlof{\lcc+1'}
  \quad\text{ for every } \sigma \geq 0 \; .
\end{equation*}


Let $\sect$ and $\sect_0$ be canonical sections of $S$ and $S_0$
respectively. 
Write $S = \sum_{i \in \Iset||} S_i$, where each $S_i$ is irreducible,
and let $\sect_i$ be a canonical section of $S_i$ such that $\sect
= \prod_{i \in \Iset||} \sect_i$.
Set also, for every $p \in \Iset$ (i.e.~for every $\sigma$-lc centre $\lcS \subset \lcc'$),
\begin{equation*}
  S\ps
  := \;\;
  \smashoperator{ \sum_{i \in \Iset|| \colon S_i \not\supset \lcS} } S_i \; ,
  \quad
  \sect\ps
  := \;\;
  \smashoperator{
    \prod_{i \in \Iset|| \colon S_i \not\supset \lcS} }\; \sect_i
  \quad\text{ and }\quad
  \psi\ps := \phi\ps -\sm\vphi\ps := \phi_{S\ps} -\sm\vphi_{S\ps}
  =\psi_{S\ps} \; ,
\end{equation*}
where $\psi_{S\ps}$, $\sm\vphi_{S\ps}$ and $\phi_{S\ps}
:=\log\abs{\sect\ps}^2$ are defined as in Notation
\ref{notation:potentials}.
(Notice that the ideal sheaf $\defidlof{\lcc+1'}$ can be generated by
$\set{\sect\ps}_{p \in \Iset}$ by treating each $\sect\ps$ as a local
function.) 
Furthermore, for any $\nu := \paren{\nu_i}_{i \in \Iset||}$ with
$\nu_i \in \fieldR_{\geq 0}$, write
\begin{equation*}
  \nu \cdot S := \sum_{i\in\Iset||} \nu_i S_i \; , \quad
  \nu \cdot S\ps
  := \;\;
  \smashoperator{\sum_{i\in\Iset|| \colon S_i \not\supset \lcS}}
  \nu_i S_i
  \quad\text{ and }\quad
  \psi\ps[\nu]. := \psi_{\nu \cdot S\ps} \; .
\end{equation*}

Let $\Diff_p S$ be the \emph{general different of $S$ on $\lcS$} (see
\cite{Kollar_Sing-of-MMP}*{\S 4.2}), given by
\begin{equation*}
  K_{\lcS} \otimes \Diff_p S = \parres{K_X \otimes S}_{\lcS} \; .
\end{equation*}
Note that $\Diff_p S = \res{S\ps}_{\lcS}$ and $\res{\sect\ps}_{\lcS}$
is a canonical section of $\Diff_p S$.
Sometimes, $\sect\ps$, $\psi\ps$ and $\psi\ps[\nu].$ may be abused to mean their
restrictions to $\lcS$.

Let $B$ be a reduced snc divisor on $X$ which has no common components
with $S$ and that $B + S$ has only snc.
Assume from now on that
\begin{equation*}
  P_{\bphi} \leq B \; .
\end{equation*}

Write also $B = \sum_{j \in I_B} B_j$ as the
decomposition into irreducible components.
Let $\sect[b]_j$ be a canonical section of $B_j$ and
$\phi_{B_j} := \log\abs{\sect[b]_j}^2$ for $j \in I_B$.
For any $\mu =\paren{\mu_j}_{j \in I_B}$ with $\mu_j \in \fieldR_{\geq
0}$, define $\mu \cdot B$ analogously to $\nu \cdot S$ and define
$\psi_{\mu\cdot B}$ as in Notation \ref{notation:potentials}.

For later use, it is remarked here that any quasi-psh function like
$\psi$ can be decomposed into 
\begin{equation} \label{eq:decompose-psi_S+P}
  \psi =
  \underbrace{
    \sum_{i \in \Iset||} \nu_i \paren{\phi_{S_i} -\sm\vphi_{S_i}}
  }_{\displaystyle =:  \psi_{\nu \cdot S} \mathrlap{\ \leq -1}}
  +\underbrace{
    \sum_{j \in I_B} \mu_j \paren{\phi_{B_j} -\sm\vphi_{B_j}}
  }_{\displaystyle =: \psi_{\mu \cdot B} \mathrlap{\ \leq -1}} \; ,
\end{equation}
after choosing the smooth potentials $\sm\vphi_{S_i}$ and
$\sm\vphi_{B_j}$ suitably.
Notice that the polar set of $\psi$ is thus split into $\psi_{\nu\cdot
  S}^{-1}(-\infty) \subset S$ and $\psi_{\mu \cdot B}^{-1}(-\infty)
\subset B$.
Note also that $\nu_i = \lelong{\psi}[S_i] =\lelong{\psi_{\nu\cdot
    S}}[S_i]$, the generic Lelong numbers of $\psi$ (or
$\psi_{\nu\cdot S}$) along $S_i$,\footnote{
  The Lelong number $\lelong{\vphi}[x]$ used here is normalised such
  that $\lelong{\log\abs{z}^2}[0] = 1$, where $z$ is a holomorphic
  coordinate function (i.e.~$\frac 12$ of the usual Lelong number
  defined in, say, \cite{Demailly_multiplier-ideal-sheaves}*{\S 2.B}),
  for convenience.
  The \emph{generic Lelong number} of $\vphi$ along an irreducible
  divisor $D$ is given by $\lelong{\vphi}[D] := \inf_{x \in D}
  \lelong{\vphi}[x]$.
} for each $i \in \Iset||$, and,
similarly, $\mu_j = \lelong{\psi}[B_j] =\lelong{\psi_{\mu\cdot
    B}}[B_j]$ for each $j \in I_B$.


Let $V \subset X$ be a sufficiently small Stein open set in a
coordinate chart (more precisely, an admissible open set; see
\cite{Chan_adjoint-ideal-nas}*{Sec.~4.1}). 
By \cite{Chan_adjoint-ideal-nas}*{Thm.~4.1.2 (2)}, for any $f \in
\aidlof*(\cl V)$ with domain some neighbourhood $V' \Supset V$ and for
any compactly supported smooth cut-off function $\rho \colon V' \to
[0,1]$ with $\res\rho_V \equiv 1$, one has 
\begin{equation} \label{eq:RTF-at-0}
  \begin{aligned}
    \eps \int_{V'} \frac{\rho \abs f^2 \:e^{-\vphi_L-\psi}}{\logpole}
    \xrightarrow{\eps \tendsto 0^+}
    &~\sum_{p \in \Iset}
    \frac{\pi^\sigma}{(\sigma -1)! \:\vect\nu_p}
    \int_{\lcS \cap  V'}
    \rho \abs{\frac{f}{\sect_0 \sect\ps}}^2
    e^{-\bphi} \dvol_{\lcS}
    \\
    \xrightarrow{\rho \descendsto \charfct_{\cl V}}
    &~\sum_{p \in \Iset}
    \frac{\pi^\sigma}{(\sigma -1)! \:\vect\nu_p}
    \int_{\lcS \cap  V}
    \abs{\frac{f}{\sect_0 \sect\ps}}^2
    e^{-\bphi} \dvol_{\lcS}
    \\
    =: &~\sum_{p \in \Iset}
    \norm{
      \frac{f}{\sect_0 \sect\ps}
    }_{\mathrlap{\lcS \cap  V}}^2 \quad
    =: \norm{\Res(f)}_{\lcc' \cap V}^2 < +\infty \; ,
  \end{aligned}
\end{equation}
where $\rho \descendsto \charfct_{\cl V}$ denotes the limit that $\rho$
descends pointwisely to the characteristic function $\charfct_{\cl V}$
of $\cl V$ on $X$, $\vect\nu_p$ is the product of all generic Lelong numbers
$\lelong{\psi}[S_i]$ along $S_i$ which contains $\lcS$.
Note that $\res{\frac{f}{\sect_0 \sect\ps}}_{\lcS}$ is
holomorphic.
The finiteness of the far right-hand side indicates that there is a
well-defined homomorphism
\begin{equation*}
  \xymatrix@C=2.2em@R=.3em{
    {\aidlof}
    \ar[r]^-{\Res}
    \ar@{}[d]|(0.4)*[left]{\in}
    & *++{\residlof{\vphi_L}[\psi]}
    \ar@{}[d]|(0.4)*[left]{\in}
    & *+<-1.5em,0pt>{\smash[b]{\bigoplus_{p \in \Iset}}\:
      \res{ S_0^{-1}}_{\lcS} \otimes
        \paren{\Diff_p  S}^{-1} \otimes
        \mtidlof<\lcS>{\bphi}}
      \ar@{}[l]|(0.7)*+{:=}
    \\
    *++{ f} \ar@{|->}[r]
    & {\paren{\res{\frac{ f}{\sect_0
            \sect\ps}}_{\lcS}}_{\mathrlap{p\in\Iset}} \mathrlap{\qquad ,}} 
  }
\end{equation*}
which is called the \emph{residue morphism}, while the sheaf
$\residlof* := \residlof{\vphi_L}[\psi]$ is called the \emph{residue
  sheaf}.
(Note that it is well defined even for $\sigma = 0$, in which case
$\residlof|0|* = S_0^{-1} \otimes S^{-1} \otimes \mtidlof<X>{\bphi}$.)
Apparently, the sheaf $\residlof*$ is supported on $\lcc'$ (and is
actually the push-forward of a sheaf living on the normalisation of
$\lcc'$).
The norm $\norm\cdot_{\lcc'}$ given in \eqref{eq:RTF-at-0} is referred
to as the \emph{residue norm} on $\residlof*$ over $\lcc'$.
As explained in \cite{Chan_adjoint-ideal-nas}*{Sec.~4.2}, when the
residue morphism $\Res$ is twisted by $K_{X}$, the residue
morphism is given by
\begin{equation*}
  \xymatrix@C=2.2em@R=.3em{
    {K_{ X} \otimes {\aidlof*}}
    \ar[r]^-{\Res}
    \ar@{}[d]|(0.4)*[left]{\in}
    & *++{K_{ X} \otimes {\residlof*}}
    \ar@{}[d]|(0.4)*[left]{\in}
    & *+<-3em,0pt>{\smash[b]{\bigoplus_{p \in \Iset}}\:
      K_{\lcS} \otimes \parres{ S_0^{-1} \otimes S^{-1}}_{\lcS} \otimes
      \mtidlof<\lcS>{\bphi}}
    \ar@{}[l]|(0.77)*+{=}
    \\
    *++{ f} \ar@{|->}[r]
    & {}\save +<2em,0pt>{\paren{\PRes[\lcS](\frac{ f}{\sect_0
            \sect})}_{\mathrlap{p\in\Iset}} \mathrlap{\qquad ,}} \restore
  }
\end{equation*}
where $\PRes[\lcS]$ denotes the Poincaré residue map with images
on $\lcS$ (see \cite{Kollar_Sing-of-MMP}*{\S 4.18}).

The theorem in \cite{Chan_adjoint-ideal-nas}*{Thm.~4.3.1} (proved
essentially via a local $L^2$ extension theorem) guarantees that the
residue morphism fits into the \emph{residue short exact sequence}
\begin{equation} \label{eq:residue-seq}
  \xymatrix{
    {0} \ar[r]
    & {{\aidlof-1*}} \ar[r]
    & {{\aidlof*}} \ar[r]^-{{\Res}}
    & {{\residlof*}} \ar[r]
    & {0 \; .}
  }
\end{equation}




{
  \setDefaultvphi{\bphi}
  \setDefaultMetric{\rs*\omega}


  \NewDocumentCommand{\vphilist}{
  }{\defaultvphi, \defaultMetric}

  \subsection{Harmonic forms under semi-positive curvature}
  \label{sec:harmonic-forms-pos-curv}


{

  \NewDocumentCommand{\decor}{
    O{\dbar}               
    D//{\bullet,\bullet}   
  }{{\vphantom{#1}}^{#2}}

  Set $X^\circ := X \setminus B$ and $\psi_{B} :=
  \phi_{B} -\sm\vphi_{B} \leq -1$ (see Notation \ref{notation:potentials}).
  As in \cite{Chan&Choi_injectivity-I}*{\S 2.2, Item (4)}, let
  \begin{equation*}
    \rs\omega := 2\omega + \ibddbar
    \frac{1}{\log\abs{\ell\psi_{B}}}
    \quad\text{ for } \ell \gg e
  \end{equation*}
  be a \emph{complete} Kähler metric on $X^\circ$ (note that it is not
  complete on $X^\circ \setminus S$) such that $\rs\omega \geq \omega$ (where
  $\omega$ is a smooth Kähler metric on $X$).
  Note that $\rs\omega$ has bounded local potential around every point
  in $X$ (not only $X^\circ$).
  Let $\Ltwo/n,q/<X>{L'} := \Ltwo/n,q/<X>{L'}_{\vphilist}$ be the Hilbert
  space of $L'$-valued $(n,q)$-forms with $L^2$
  coefficients with respect to $\bphi$ and $\rs\omega$ on
  $X^\circ$.\footnote{
    The notation ``$\Ltwo/n,q/<X>{L'}$'' follows the convention in
    \cite{Chan&Choi&Matsumura_injectivity-II}.
    See \cite{Chan&Choi&Matsumura_injectivity-II}*{footnote 1}.
  } 
  Note that the exterior differential operator $\dbar$ acting on
  currents is a densely defined operator with closed graph on
  $\Ltwo/n,q/<X>{L'}$ with domain given by
  \begin{equation*}
    \paren{\Dom\dbar}\decor/n,q/ :=\paren{\Dom\dbar}\decor/n,q/_{\vphilist} 
    := \setd{
      \zeta \in \Ltwo/n,q/<X>{L'}_{\vphilist}
    }{
      \dbar\zeta \in \Ltwo/n,q+1/<X>{L'}_{\vphilist}
    } \; .
  \end{equation*}
  Denote the corresponding kernel and image of $\dbar$ by
  $\paren{\ker\dbar}\decor/n,q/$ and
  $\paren{\im\dbar}\decor/n,q/$ respectively (the superscripts
  indicate that they are subspaces of $\Ltwo/n,q/<X>{L'}$).
  The $L^2$ Dolbeault isomorphism on $X$ (see
  \cite{Fujino_injectivity-II}*{Lemma 3.20} or
  \cite{Matsumura_injectivity-Kaehler}*{Prop.~2.16}) states that
  \begin{equation*} 
    \cohgp q[X]{K_X \otimes L' \otimes \mtidlof{\bphi}}
    \isom
    \frac{\paren{\ker\dbar}\decor/n,q/}{\paren{\im\dbar}\decor/n,q/}
    =: \cohgp{n,q}<\dbar>[X]{L'}_{\vphilist} \; ,
  \end{equation*}
  in which the left-hand side is treated as the \v Cech cohomology group.

  With $\dbadj$ denoting the Hilbert space adjoint of $\dbar$ with
  respect to the $L^2$ norm $\norm{\cdot}_{X^\circ} :=
  \norm\cdot_{X^\circ,\vphilist}$, one has the orthogonal
  decomposition (see, for example, \cite{Demailly}*{Ch.~VIII, \S 1})
  \begin{equation*}
    \Ltwo/n,q/{L'}_{\vphilist}
    = \Harm' \oplus \cl{\paren{\im\dbar}}\decor/n,q/
    \oplus \cl{\paren{\im\dbadj}}\decor/n,q/
    = \Harm' \oplus \paren{\im\dbar}\decor/n,q/
    \oplus \paren{\im\dbadj}\decor/n,q/ \; ,
  \end{equation*}
  where $\Harm' := \Harm'{L'},_{\rs\omega} :=
  \Harm{\logKX<X>[L']},_{\rs\omega}$ is the space of $L'$-valued
  harmonic $(n,q)$-forms with respect to $\bphi$ and $\rs\omega$ on
  $X^\circ$, and $\paren{\im\dbar}\decor/n,q/$ and
  $\paren{\im\dbadj}\decor/n,q/$ are the images of $\dbar$ and
  $\dbadj$ respectively in $\Ltwo/n,q/{L'}$, which are closed as $X$
  is compact (see, for example,
  \cite{Matsumura_injectivity}*{Prop.~5.8}).
  One therefore obtains the isomorphisms
  \begin{equation*} 
    \Harm' \isom \cohgp{n,q}<\dbar>[X]{L'}_{\vphilist}
    \isom
    \cohgp q[X]{K_X \otimes L' \otimes \mtidlof{\bphi}} \; .
  \end{equation*}

  With a finite Stein open cover $\cvr V = \set{V_i}_{i \in I}$,
  Leray's theorem assures the isomorphism $\cohgp q[\cvr V]{\dotsm}
  \isom \cohgp q[X]{\dotsm}$ as soon as the coefficient space
  ($\dotsm$) is a coherent sheaf.
  Fix a partition of unity $\set{\rho^i}_{i \in I}$ subordinate to
  $\cvr V$.
  For every harmonic form $u \in \Harm'$, one can solve
  $\dbar$-equations induced from $u$ using H\"ormander's $L^2$ estimates on the
  intersections of the Stein open sets in $\cvr V$ (see
  \cite{Matsumura_injectivity}*{Prop.~5.5} or
  \cite{Chan&Choi_injectivity-I}*{Lemma 3.2.1 and Remark 3.2.2}) to
  obtain a $K_X \otimes L' \otimes \mtidlof{\bphi}$-valued \v Cech
  cocycle $\seq{\alpha_{\idx 0.q}}_{\idx 0,q \in I}$ with respect to
  $\cvr V$ such that (under the Einstein summation convention)
  \begin{equation} \label{eq:Cech-Dolbeault-isom}
    \begin{aligned}
      u &=\dbar v_{(2)} +\dbar \rho^{i_{q-1}} \wedge \dotsm \wedge
      \dbar\rho^{i_0} \alpha_{\idx 0.q} \qquad\paren{\forall~ i_q \in
        I}
      \\
      &=\dbar v_{(2)} +\dbar \rho^{i_{q-1}} \wedge \dotsm \wedge
      \dbar\rho^{i_0} \cdot \rho^{i_q} \:\alpha_{\idx 0.q}
      \\
      &=\dbar v_{(2)} +(-1)^q \:\underbrace{\dbar \rho^{i_{q}} \wedge
        \dotsm \wedge \dbar\rho^{i_1} \cdot \rho^{i_0} }_{=: \:
        \paren{\dbar\rho}^{\idx q.0}} \alpha_{\idx 0.q}
    \end{aligned}
  \end{equation}
  for some $K_X \otimes L'$-valued
  $(0,q-1)$-form $v_{(2)}$ which is $L^2$ on $X^\circ$ with respect to $\vphilist$.
  Notice that $\seq{\alpha_{\idx 0.q}}_{\idx 0,q \in I}$ represents
  the class of $u$ in $\cohgp q[X]{K_X \otimes L' \otimes
    \mtidlof{\bphi}}$ under the $L^2$ Dolbeault isomorphism. 

  The Bochner--Kodaira formula on 
  $X$, together with the positivity $\ibddbar\bphi \geq 0$, assures
  the following.

  \begin{prop}[\cite{Chan&Choi&Matsumura_injectivity}*{Prop.~2.2.2}]
    \label{prop:consequence-of-positivity}
    
    Suppose that
    $\ibddbar\bphi \geq 0$ and $u \in \Harm'{L'},_{\rs\omega}$.
    Then, one has
    \begin{equation*}
      \nabla^{(0,1)}u = 0 \quad\text{ and }\quad
      \idxup{\ibddbar\bphi} \ptinner{u}{u}_{\rs\omega} = 0
      \quad\text{ on } X^\circ \; .
    \end{equation*}
  \end{prop}

  Using the cut-off $\theta_\eps := \theta \circ
  \frac{1}{\abs{\psi_S}^\eps}$ in the proof of
  \cite{Chan&Choi&Matsumura_injectivity}*{Lemma 2.2.1} (where $\psi_S
  := \phi_S - \sm\vphi_S$ and $\theta \colon [0,\infty) \to [0,1]$ is
  a smooth non-decreasing function such that
  $\res{\theta}_{[0,\frac 12]} \equiv 0$ and
  $\res{\theta}_{[1,\infty)} \equiv 1$), one can obtain the following.
  
  \begin{lemma}[\cite{Chan&Choi&Matsumura_injectivity}*{Lemma 2.2.1}]
    \label{lem:su-harmonicity}
    If $u \in \Harm'{L'},_{\rs\omega}$, then $u \sect \in
    \Harm'{L' \otimes S},{\bphi+\phi_S}_{\rs\omega}$ (and $u \sect_0 \sect \in
    \Harm'{L},{\vphi_L+\psi}_{\rs\omega}$). 
  \end{lemma}

  Let $\psi_\bullet$ be either $\psi_{\nu \cdot S}$ or $\psi_{\mu
    \cdot B}$, and $\sm\vphi_\bullet$ be $\sm\vphi_{\nu \cdot S}$ or
  $\sm\vphi_{\mu \cdot B}$ (defined as in Notation
  \ref{notation:potentials}) accordingly.
  Using again the cut-off function $\theta_{\eps'}$ with $\eps' > 0$, one can see that the
  twisted Bochner--Kodaira formula, with the
  twisting function $\abs{\psi_\bullet}^{1-\eps}$ for any $\eps > 0$, is
  valid for any harmonic form in $\Harm'{L'},_{\rs\omega}$ (see the
  proof of 
  \cite{Chan&Choi_injectivity-I}*{Prop.~3.2.8} or 
  \cite{Chan&Choi&Matsumura_injectivity-II}*{Prop.~3.2.6}).
  The following equality is obtained from the twisted Bochner--Kodaira
  formula together with the conclusion of Proposition
  \ref{prop:consequence-of-positivity}.

  \begin{prop}[\cite{Chan&Choi_injectivity-I}*{Prop.~3.2.8} or 
    \cite{Chan&Choi&Matsumura_injectivity-II}*{Prop.~3.2.6}]
    \label{prop:consequence-of-tBK}

    Under the assumption that $\ibddbar\bphi \geq 0$, every harmonic form $u \in
    \Harm'{L'},_{\rs\omega}$ satisfies
    \begin{equation*}
      \lim_{\eps \tendsto 0^+} \eps \int_{\mathrlap{X^\circ}} \;\;
      \frac{
        \abs{\idxup{\diff\psi_\bullet} . u}_{\vphilist}^2
      }{\abs{\psi_\bullet}^{1+\eps}}
      = \int_{X^\circ} \idxup{\ibddbar\sm\vphi_\bullet}
      \ptinner{u}{u}_{\vphilist}
      \; .
    \end{equation*}
    The limit on the left-hand side therefore exists (and is thus finite)
    because $u$ is $L^2$ with respect to $\vphilist$ on $X^\circ$ and
    $\ibddbar\sm\vphi_\bullet$ is a smooth form on $X$.
  \end{prop}

  

  The singularities of $u \in \Harm'{L'}$ along $X \setminus X^\circ =
  B$ can be given as follows. 

  \begin{prop}[\cite{Matsumura_injectivity-lc}*{Thm.~3.3} and
    \cite{Chan&Choi_injectivity-I}*{Thm.~2.5.1 and Prop.~3.3.1 and 3.3.2}]
    \label{prop:singularities-along-FM}
    
    Given that $\ibddbar\bphi \geq 0$ and $u \in \Harm'{L'}$, it
    follows that $*u$ is holomorphic on $X$, where $*$ is the Hodge
    $*$-operator with respect to $\rs\omega$.
    This implies that
    \begin{equation*}
      \text{coef.~of } u
      \text{ and } \idxup{\diff\psi_{\nu\cdot S}} . u \sect
      \in \smooth_X \left[
        \abs{\psi_{B}}^{\pm} \: ,
        \:\paren{\log\abs{\ell \psi_{B}}}^{\pm} \: ,
        \:\frac 1{\abs{\sect[b]_j}} \colon  j \in I_{B}
      \right] \quad\text{ on } X \; .
    \end{equation*}
  \end{prop}

  As $\lcdata<X>[\phi_S]$ is in the snc configuration and $\bphi$ contains no
  lc centres of $(X,S)$, the singularities of $u$ along $X \setminus
  X^\circ$ given in Proposition \ref{prop:singularities-along-FM} does
  not affect the residue computation on the lc centres of $(X,S)$ in
  view of Fubini's theorem (see
  \cite{Chan&Choi_injectivity-I}*{Thm.~2.6.1 and pf.~of
    Prop.~3.3.2}).
  Let $\smooth_{X\: c}$ denote the sheaf of germs of smooth functions
  on $X$ with compact support and let the ad hoc notation
  ``$\smooth_X[\rs\omega^{\pm}]$'' temporarily denote the sheaf of
  algebra in Proposition \ref{prop:singularities-along-FM} for
  convenience.
  Further let ``$\smooth_X[\rs\omega^{\pm}] * \mtidlof{\bphi}$'' denote
  the multiplier ideal sheaf defined in $\smooth_X[\rs\omega^{\pm}]$
  in place of $\holo_X$, i.e.~the ideal sheaf of germs of functions in
  $\smooth_X[\rs\omega^{\pm}]$ which are locally $L^2$ with respect to
  $\bphi$.
  One has the following local residue formula.

  \begin{prop}[\cite{Chan&Choi&Matsumura_injectivity}*{Prop.~2.3.1 and
      2.3.2}]
    \label{prop:residue-formula-classical-kernel}
    
    Given any open set $V \subset X$ and any compactly
    supported section $f \in \logKX.[L'] \otimes \smooth_{X
      \:c}[\rs\omega^{\pm}] * \aidlof|1|{\bphi}[\psi_S] \paren{V}$
    such that $\Res^1(f) =
    \paren{\PRes[\lcS|1|[b]](\frac{f}{\sect})}_{b \in \Iset|1|} =:
    \paren{g_b}_{b\in\Iset|1|}$, one has, 
    for any $\xi \in \logKX.[L'] \otimes \smooth_{X}[\rs\omega^{\pm}]
    * \mtidlof{\bphi} \paren{V}$,
    \begin{equation*}
      \left.
        \begin{aligned}
          \lim_{\eps \tendsto 0^+} \eps \int_V \inner{\xi}{f}
          \:e^{-\phi_S-\bphi} e^{-\eps\abs{\psi_{\nu \cdot S}}} &
          \\
          \lim_{\eps \tendsto 0^+} \eps \int_V \frac{
            \inner{\xi}{f}
            \:e^{-\phi_S-\bphi}
          }{\abs{\psi_{\nu \cdot S}}^{1+\eps}}  &
        \end{aligned}
      \right\}
      =\sum_{b \in \Iset|1|} \frac{\pi}{\nu_b}
        \int_{\lcS|1|[b] \cap V} \inner{\frac{\rs*\xi_b}{\sect_{(b)}}}{\: g_b}
        \:e^{-\bphi} \;\;  < \infty \; ,
    \end{equation*}
    where
    $\nu_b := \lelong{\psi_{\nu \cdot S}}[\lcS|1|[b]]$ and
    \begin{equation*}
      \rs*\xi_b := \PRes[\lcS|1|[b]](\frac{\xi}{\sect}) \cdot \sect_{(b)}
      \in K_{\lcS|1|[b]} \otimes \Diff_b S \otimes \res{L'}_{\lcS|1|[b]} \otimes
      \smooth_{\lcS|1|[b]\:c}[\rs\omega^{\pm}] *
      \mtidlof<\lcS|1|[b]>{\bphi} \paren{\lcS|1|[b] \cap V} \; .
    \end{equation*}    
  \end{prop}
  \begin{proof}[Remarks on the proof]
    The proofs of both statements are the same as those of
    \cite{Chan&Choi&Matsumura_injectivity}*{Prop.~2.3.1 and 2.3.2} with some
    extra care taken for the generic Lelong numbers $\nu_p$ and the
    relaxation on the coefficients of $f$ and $\xi$ from $\smooth_X$
    to $\smooth_X [\rs\omega^{\pm}]$.
    This change does not affect the residue computation in view of Fubini's
    theorem.

    Note also that, in
    \cite{Chan&Choi&Matsumura_injectivity}*{Prop.~2.3.1 and 2.3.2}, the
    coefficients are allowed to be in $\smooth_{X\,*}$ (locally bounded germs in
    $\smooth_X\left[\frac 1{\abs{\sect_i}} \colon  i\in \Iset||
    \right]$, see \cite{Chan&Choi&Matsumura_injectivity}*{\S 2.3 and
      footnote 2}).
    It is convenient to consider such coefficients when handling
    residues of $\dbar\psi_S \otimes u$ (see
    \cite{Chan&Choi_injectivity-I}*{proof of Prop.~3.2.3}), but it is
    not necessary when dealing with residues of
    $\idxup{\diff\psi_S}. u$ or $\idxup{\diff\psi_{\nu\cdot S}}. u$, as we do in this paper.
  \end{proof}

  Note that all statements above and in previous sections for
  $(X,S)$ are applicable to $(\lcS, \Diff_p S)$ for any $p \in \Iset$.

}


  \subsection{Harmonic residues}
  \label{sec:harmonic-residue}


Keep writing $\aidlof* := \aidlof{\vphi_L}$, $\residlof* :=
\residlof{\vphi_L}$ and
\begin{equation*}
  \spH{\sheaf F} :=\cohgp q[X]{K_X\otimes L \otimes \sheaf F}
  \quad\text{ for any sheaf } \sheaf F \text{ on } X
\end{equation*}
for convenience.
Recall that, taking into account the $L^2$ Dolbeault isomorphism and
harmonic theory, one has
\begin{equation*}
  \spH{\residlof*} \isom \bigoplus_{p \in \Iset}
  \Harm<\lcS>{\logKX<\lcS>[L']}
  =: \spHarm{\residlof*}  \; .
\end{equation*}
For $q \geq 1$, let
\begin{equation*}
  \delta := \delta^{q-1} \colon
  \spH/q-1/{\residlof+1*} \to \spH{\residlof*}
\end{equation*}
be the connecting morphism in the long exact sequence induced from the
short exact sequence
\begin{equation*}
  \renewcommand{\objectstyle}{\displaystyle}
  \xymatrix@R=3.5ex{
    {0} \ar[r]
    &{\faidlof/-1*} \ar[r] \ar[d]_-{\Res^{\sigma}}^-{\isom}
    &{\faidlof+1/-1*} \ar[r]
    &{\faidlof+1/*} \ar[r] \ar[d]^-{\Res^{\sigma+1}}_-{\isom}
    &{0 \; .}
    \\
    &{\residlof*} 
    &&{\residlof+1*} 
  }
\end{equation*}
The image of $\delta$ plays a role in the proof of the
injectivity theorem (see \cite{Chan&Choi&Matsumura_injectivity}*{Proof
  of Thm.~3.4.1, Step IV}).
To compute it in terms of \v Cech cohomology, 
note that $\delta$ is induced from the \v Cech coboundary
operator.
For each $w =\paren{w_b}_{b\in \Iset+1} \in
\spHarm/q-1/{\residlof+1*}$, write
\begin{equation} \label{eq:Cech-Dolbeault-on-w_b}
  w_b \quad\;
  \overset{\mathclap{\text{\eqref{eq:Cech-Dolbeault-isom}}}}= \quad\;
  \dbar v_{b;(2)} + (-1)^{q-1} \frac{v_{b;(\infty)}}{\sect_{(b)}}
  \quad\text{ on } \lcS+1[b]' := \lcS+1[b] \cap X^\circ \; ,
\end{equation}
where
\begin{equation*}
  \frac{v_{b;(\infty)}}{\sect_{(b)}}
  := \paren{\dbar\rho}^{\idx q.1}
  \frac{\rs\gamma_{b;\:\idx 1.q}}{\sect_{(b)}}
  := \paren{\dbar\rho}^{\idx q.1} \alpha_{b;\:\idx 1.q}
  \quad\text{ on } \lcS+1[b] 
\end{equation*}
such that $v_{b;(2)} \in
\Ltwo/0,q-2/<\lcS+1[b]>{\logKX<\lcS+1[b]>[L']}_{\vphilist}$ and
$\set{\alpha_{b;\:\idx 1.q}}_{\idx 1,q \in I}$ is a \v Cech
$(q-1)$-cocycle representing $w_b$ with respect to the cover $\cvr V
\cap \lcS+1[b]$ (each $\alpha_{b;\:\idx 1.q}$ is holomorphic and is
(globally) $L^2$ with respect to $\bphi$ and $\rs\omega$ on
$\lcS+1[b] \cap V_{\idx 1.q}$).
(Write $\alpha_{b;\:\idx 1.q} =: \frac{\rs*\gamma_{b;\:\idx
    1.q}}{\sect_{(b)}}$ just to have a consistent notation with
\cite{Chan&Choi&Matsumura_injectivity}*{Prop.~2.3.3}.)
Choosing the liftings $\gamma_{\idx 1.q} \in \logKX[L] \otimes
\aidlof+1<X>*\paren{V_{\idx 1.q}}$ such that
\begin{equation*}
  \PRes[\lcS+1[b]](\frac{\gamma_{\idx 1.q}}{\sect_0 \sect})
  =
  \frac{\rs\gamma_{b;\:\idx 1.q}}{\sect_{(b)}} \quad\text{ for all } b
  \in \Iset+1 \; ,
\end{equation*}
the component of the image $\delta w$ on $\lcS$
can be represented by the \v Cech $q$-cocycle $\frac{\delta\rs\gamma_p}{\sect_{(p)}}$,
where
\begin{equation*}
  \frac{\rs\gamma_{p;\:\idx 1.q}}{\sect_{(p)}} :=
  \PRes[\lcS](\frac{\gamma_{\idx 1.q}}{\sect_0 \sect})
\end{equation*}
and $\delta\rs\gamma_p$
denotes the image of the cochain $\set{\rs\gamma_{p;\:\idx 1.q}}_{\idx
  1,q \in I}$ under the \v Cech coboundary operator.
In view of the \v Cech--Dolbeault map \eqref{eq:Cech-Dolbeault-isom},
the image $\delta w$ is then represented by $\paren{-\frac{\dbar
    v_{p;(\infty)}}{\sect_{(p)}}}_{p \in \Iset}$, where (see
\cite{Chan&Choi&Matsumura_injectivity}*{Proof of Thm.~3.4.1, Step IV})
\begin{equation*}
  - \frac{\dbar v_{p;(\infty)}}{\sect_{(p)}}
  :=
  - \frac{
    \dbar\paren{\paren{\dbar\rho}^{\idx q.1} \rs\gamma_{p;\:\idx 1.q}}
  }{\sect_{(p)}}
  = (-1)^q \paren{\dbar\rho}^{\idx q.0}
  \frac{\paren{\delta \rs*\gamma_p}_{\idx 0.q}}{\sect_{(p)}}
  \quad\text{ on } \lcS \; .
\end{equation*}
Note that this is $\dbar$-closed and $L^2$ with respect to $\bphi$ and
$\rs\omega$ on $\lcS \cap X^\circ$.
Denote by $\deltaH w$ the orthogonal projection to
$\spHarm{\residlof*}$ of $\paren{(-1)^q \frac{\dbar
    v_{p;(\infty)}}{\sect_{(p)}}}_{p\in \Iset}$, a Dolbeault
representative of $(-1)^{q-1} \delta w$.
Thus $\deltaH w$ is the harmonic representative of
$(-1)^{q-1} \delta w$.


Recall from
\cite{Chan&Choi&Matsumura_injectivity}*{Prop.~2.3.3} that, for any
$\sigma$-lc centre $\lcS$ and $(\sigma+1)$-lc centre $\lcS+1[b]$ such
that $\lcS+1[b] \subset \lcS$, the sign $\sgn{b:p}$ is defined by
\begin{equation*}
  \PRes[\lcS+1[b]] =\sgn{b:p} \:\PRes[\lcS+1[b] | \lcS] \circ
  \PRes[\lcS] \; ,
\end{equation*}
where $\PRes[\lcS+1[b] | \lcS]$ denotes the Poincar\'e residue map
from $\lcS$ to $\lcS+1[b]$.
Let $\nu_p \cdot S\ps$ be a divisor defined as in Section
\ref{sec:notation-in-snc} such that $\supp \parres{\nu_p \cdot
  S\ps}_{\lcS} = \Diff_p S$, and let $\psi\ps. :=\psi_{\nu_p \cdot S\ps}$ be
the function defined as in Notation \ref{notation:potentials}.
Given any collection of $u :=(u_p)_{p \in \Iset}$ of harmonic forms
$u_p \in \Harm<\lcS>{\logKX<\lcS>[L']}$ on $\lcS$ for each $p \in \Iset$, define
\begin{equation} \label{eq:definition-of-w}
  \HRes(u) := \paren{\HRes(u)_b}_{b \in \Iset+1}  \; ,
\end{equation}
where
\begin{equation*}
  \HRes(u)_b := \smashoperator[r]{
    \sum_{p\in\Iset \colon \lcS+1[b] \subset \lcS}
  } 
  \: \sgn{b:p} \:
  \PRes[\lcS+1[b] | \lcS](\idxup{\diff\psi\ps.}.  u_{p})
  \quad\text{ on $\lcS+1[b]$ for each } b \in \Iset+1 \; .
\end{equation*}
The collection $\HRes(u)$ is referred to as the
\emph{harmonic residue} of $u$ on $\lcc+1'$. 
If $\lcS+1[b] \cap V = \set{z_b = 0} \cap \lcS \cap V$ for some
coordinate function $z_b$ on some admissible open set $V$, one has
\begin{equation*}
  \PRes[\lcS+1[b] | \lcS](\idxup{\diff\psi\ps.}. u_p)
  =\PRes[\lcS+1[b] | \lcS](\idxup{\diff\phi\ps.}. u_p)
  = \nu_{b:p} \PRes[\lcS+1[b] | \lcS](\idxup{\diff\log\abs{z_b}^2}. u_p) \; ,
\end{equation*}
where $\nu_{b:p} := \lelong{\res{\psi\ps.}_{\lcS}\:}[\lcS+1[b]]$
(which is computed in $\lcS$).
This shows that \emph{the dependence of $\HRes(u)$ on the functions $\psi\ps.$
lies only on the choices of the coefficients $\nu_p$ in $\psi\ps.$'s}.


\begin{prop}[\cite{Chan&Choi&Matsumura_injectivity}*{Thm.~2.4.3 and
    Prop.~2.3.3}]
  \label{prop:res-formula-dbar-exact-dot-harmonic}
  For any harmonic $u := \paren{u_p}_{p \in \Iset} \in \bigoplus_{p \in \Iset}
  \Harm<\lcS>{\logKX<\lcS>[L']} = \spHarm{\residlof*} \isom
  \spH{\residlof*}$, under the assumption $\ibddbar\bphi \geq 0$, the
  harmonic residue $\HRes(u)$ is harmonic and represents a class in
  $\spH/q-1/{\residlof+1*}$.
  Furthermore, for any harmonic $w = \paren{w_b}_{b \in \Iset+1} \in
  \spHarm/q-1/{\residlof+1*} \isom \spH/q-1/{\residlof+1*}$, one has
  \begin{equation*}
    \iinner{\deltaH w}{u}_{\lcc<X^\circ>'}
    =\sigma_+ \iinner{w}{\HRes(u)}_{\lcc+1<X^\circ>'} \; , 
  \end{equation*}
  where $\sigma_+ :=\max\set{1,\sigma}$, $\lcc|\bullet|<X^\circ>' :=
  \lcc|\bullet|' \cap X^\circ$, and the
  inner products are given by the residue norms as in
  \eqref{eq:RTF-at-0}
  induced from $\bphi$, $\rs\omega$ and (Lelong numbers of) $\psi$.
\end{prop}

\begin{proof}

  The fact that $\HRes(u)$ is a harmonic form in
  $\spHarm/q-1/{\residlof+1*} \isom \spH/q-1/{\residlof+1*}$ is proved
  in \cite{Chan&Choi&Matsumura_injectivity}*{\S 2.4}. 
  The proof of the equality is formally the same as that of
  \cite{Chan&Choi&Matsumura_injectivity}*{Prop.~2.3.3} (see also
  \cite{Chan&Choi&Matsumura_injectivity-II}*{Thm.~5.2.1}), with an extra
  care for the singularities of $u$ along $X \setminus X^\circ = B$
  and the coefficients $\nu_p$ in the function $\psi\ps.$. 
  A sketch is given below.

  Write $\lcS' := \lcS \cap X^\circ$ for any $\sigma \geq 0$ and any
  $p\in\Iset$. 
  Set $\iinner{\cdot}{\cdot}_{\lcS||[\bullet]', \phi_{(\bullet)}}
  :=\iinner{\cdot}{\cdot \:e^{-\phi_{(\bullet)}}}_{\lcS||[\bullet]'}$
  for $\bullet = p, b$.
  Abuse $\delta w$ to mean its representative $\paren{-\frac{\dbar
      v_{p;(\infty)}}{\sect_{(p)}}}_{p \in \Iset}$.
  The smooth form $v_{p;(\infty)}$ on $\lcS$ need not be $L^2$ with
  respect to the weight $e^{-\phi_{(p)}}$, so
  an integration by parts is done via the use of Proposition
  \ref{prop:residue-formula-classical-kernel} (with the knowledge of
  the singularities of $u_p$ along $\lcS \setminus \lcS'$ by
  Proposition \ref{prop:singularities-along-FM}), which yields 
  \begin{align*}
    &~\iinner{\delta w}{u}_{\lcc<X^\circ>'}
      = \sum_{p\in\Iset} \iinner{-\frac{\dbar v_{p;(\infty)}}{\sect_{(p)}} \:}{
      \: u_p
      }_{\lcS'}
      =-\sum_{p\in \Iset} \iinner{\dbar v_{p;(\infty)}}{ u_p
      \sect_{(p)}}_{\lcS', \phi_{(p)}}
    \\
    \xleftarrow{\eps \tendsto 0^+}
    &-\sum_{p \in \Iset} \iinner{
      e^{-\eps \abs{\psi\ps.}} \:\dbar v_{p;(\infty)}
      }{ u_p \sect_{(p)}}_{\lcS', \phi_{(p)}}
    \\
    =&-\sum_{p \in \Iset} \paren{
       \begin{multlined}[c][0.7\textwidth]
         \cancelto{0 \;\;\;(\because~u_p \text{ harmonic, Lemma
             \ref{lem:su-harmonicity}%
           })}{\iinner{
             \dbar\paren{e^{-\eps \abs{\psi\ps.}} \:v_{p;(\infty)}}
           }{ u_p \sect_{(p)}}_{\mathrlap{\lcS', \phi_{(p)}}}}
         \\ 
         - \eps \iinner{
           e^{-\eps \abs{\psi\ps.}} \:v_{p;(\infty)}
         }{\:\idxup{\diff\psi\ps.} . u_p \sect_{(p)}}_{\lcS', \phi_{(p)}}
       \end{multlined}
       }
    \\
    =&~\smashoperator{\sum_{\substack{\idx 1,q \in I \, , \\ p \in \Iset}}}
    \;\;\eps \:
       \iinner{
       e^{-\eps \abs{\psi\ps.}} \:
       \rs*\gamma_{p;\:\idx 1.q}
       }{\:
       \idxup{\diff\rho},[\idx 1.q] .
       \paren{\idxup{\diff\psi\ps.} . u_p \sect_{(p)}}
       }_{\lcS', \phi_{(p)}}
    \\
    \xrightarrow[
    \substack{\text{Prop.~\ref{prop:singularities-along-FM}}
    \\ \text{Prop.~\ref{prop:residue-formula-classical-kernel}} 
    }
    ]{\eps \tendsto 0^+} 
    &~\smashoperator{\sum_{\substack{p \in \Iset \, ,\\ \idx 1,q \in I}}}
    \;\;\sum_{k=\sigma +1}^{\mathclap{\sigma_{V_{\idx 1.q}}}}
    \sigma_+
      \iinner{
      \PRes[p(k)](
      \frac{\rs*\gamma_{p;\:\idx 1.q}}{\sect_{(p)}}
      )
      }{\:
      \idxup{\diff\rho},[\idx 1.q] .
      \PRes[p(k)](\idxup{\diff\psi\ps.} . u_p)
      }_{\lcS' \cap \set{z_{p(k)} =0}}
      \; ,
  \end{align*}
  where $\idxup{\diff\rho},[\idx 1.q] . \cdot$ is the adjoint
  of $\paren{\dbar\rho}^{\idx q.1} \cdot$, and
  $\PRes[p(k)]$ denotes the Poincar\'e residue map from $\lcS$ to
  $\lcS \cap \set{z_{p(k)}=0}$, in which $(z_1, \dots, z_n)$ is a
  holomorphic coordinate system such that $\lcS \cap V_{\idx 1.q}
  =\set{z_{p(1)} = \dotsm =z_{p(\sigma)} =0}$ and $\sect_{(p)}
  =z_{p(\sigma+1)} \dotsm z_{p(\sigma_{V_{\idx 1.q}})}$.
  Note also the appearance of the coefficient $\sigma_+$, as well as
  the absence of the Lelong number $\nu_{p(k)} :=
  \lelong{\psi\ps.}[\set{z_{p(k)} = 0}]$ as appeared in Proposition
  \ref{prop:residue-formula-classical-kernel}, comes from
  the different normalisations of the $L^2$ norms on various lc centres,
  namely, $\norm\cdot_X^2 := \int_X \dotsm$ and $\norm\cdot_{\lcS}^2
  := \frac{\pi^\sigma}{(\sigma -1)! \:\vect\nu_p} \int_{\lcS} \dotsm$ for every
  integer $\sigma \geq 1$.

  Following the argument in the proof of
  \cite{Chan&Choi&Matsumura_injectivity}*{Prop.~2.3.3}, the $(\sigma
  +1)$-lc centres $\lcS \cap \set{z_{p(k)} = 0}$ for $k = \sigma +1
  , \dots, \sigma_{V_{\idx 1.q}}$ in $V_{\idx 1.q}$
  can be re-indexed in terms of $b \in \Iset+1$ such that
  \begin{equation*}
    \lcS[p_{b,j}] \cap \set{z_{b(j)} = 0}
    = \lcS+1[b] \cap V_{\idx 1.q}
    \quad\text{ for } j = 1, \dots, \sigma +1 
  \end{equation*}
  and the summations transform as $\sum_{p \in \Iset} \sum_{k=\sigma
    +1}^{\sigma_{V}} \dotsm = \sum_{b \in
    \Iset+1} \sum_{j=1}^{\sigma +1} \dotsm$.
  With such choice of indexing, one has
  \begin{equation*}
    \frac{\rs*\gamma_{b;\: \idx 1.q}}{\sect_{(b)}}
    :=\PRes[\lcS+1[b]](\frac{\gamma_{\idx 1.q}}{\sect_0 \sect})
    =\sgn{b:p_{b,j}} \:
    \PRes[b(j)](\frac{\rs*\gamma_{p_{b,j};\:\idx
        1.q}}{\sect_{(p_{b,j})}})
  \end{equation*}
  (noticing that 
  $\sect_{(p_{b,j})} = z_{b(j)} \sect_{(b)}$).
  As a result, the expression in question becomes
  \begin{align*}
    &\smashoperator[r]{\sum_{\substack{b \in \Iset+1  ,\\ \idx 1,q \in I}}}
    \;\;\sum_{j=1}^{\sigma +1}
    \sigma_+
      \iinner{ \sgn{b:p_{b,j}}\:
      \frac{\rs*\gamma_{b;\:\idx 1.q}}{\sect_{(b)}}
      }{\: 
      \idxup{\diff\rho},[\idx 1.q] .
      \PRes[b(j)](\idxup{\diff\psi\ps._{p_{b,j}}} . u_{p_{b,j}})
      }_{\lcS+1[b]'}
    \\
    =&\smashoperator[r]{\sum_{\substack{\idx 1,q \in I \, , \\ b \in \Iset+1}}}
       \;\;\sigma_+
       \iinner{
       \frac{\paren{\dbar\rho}^{\idx q.1} \rs*\gamma_{b;\:\idx 1.q}}{\sect_{(b)}}
    \:}{  \smash{\sum_{j=1}^{\sigma +1}}\:
    \sgn{b:p_{b,j}} \:
       \PRes[b(j)](\idxup{\diff\psi\ps._{p_{b,j}}} . u_{p_{b,j}})
       }_{\lcS+1[b]'}
    \\
    =&\sum_{b \in \Iset+1} \sigma_+ \iinner{
        \frac{v_{b;(\infty)}}{\sect_{(b)}}
       \:}{\quad\;
       \smash{\smashoperator{\sum_{p\in\Iset \colon \lcS+1[b] \subset
       \lcS}}} \:
       \sgn{b:p} \:
       \PRes[\lcS+1[b] | \lcS](\idxup{\diff\psi\ps.} . u_{p})
       }_{\lcS+1[b]'}
    \\
    =&\sum_{b \in \Iset+1} \sigma_+ \iinner{
        \frac{v_{b;(\infty)}}{\sect_{(b)}}
       \:}{\HRes(u)_b}_{\lcS+1[b]'}
       \overset{\text{\eqref{eq:Cech-Dolbeault-on-w_b}}}=
       \sum_{b \in \Iset+1} (-1)^{q-1} \sigma_+ \iinner{
         w_b -\dbar v_{b;(2)} \:
       }{\HRes(u)_b}_{\lcS+1[b]'}
    \\
    \overset{\mathclap{\alert[Gray]{\HRes(u)\text{ harmonic}}}}=
    &\qquad
      \sum_{b \in \Iset+1} (-1)^{q-1} \sigma_+ \iinner{
         w_b \:
      }{\HRes(u)_b}_{\lcS+1[b]'}
      =(-1)^{q-1} \sigma_+ \iinner{w}{\HRes(u)}_{\lcc+1<X^\circ>'} \; .
  \end{align*}
  Note that, since $u$ is harmonic, one has $\iinner{(-1)^{q-1} \delta
    w}{u}_{\lcc<X^\circ>'} = \iinner{\deltaH w}{u}_{\lcc<X^\circ>'}$.
  The desired equality then follows.
\end{proof}

For the application in next section, for any $\sigma' >
\sigma \geq 0$ and $q \geq 1$, consider the commutative diagram
\begin{equation} \label{eq:reduce-to-residl-comm-diagram}
  \begin{gathered}
    \xymatrix@R=1.2em{
      {\dotsm} \ar[r]
      &{\spH/q-1/{\residlof+1*}}
      \ar[r]^-{\delta}
      \ar[d]^-{\tau^{\sigma'}_{\sigma+1}}
      &{\spH{\residlof*}}
      \ar[r]^-{\tau^{\sigma+1}_{\sigma}}
      \ar@{=}[d]
      &{\spH{\faidlof+1/-1*}} \ar[r] \ar[d]_-{\iota}
      &{\dotsm}
      \\
      {\dotsm} \ar[r]
      &{\spH/q-1/{\faidlof|\sigma'|/*}}
      \ar[r]^-{\delta'}
      &{\spH{\residlof*}}
      \ar[r]^-{\tau^{\sigma'}_{\sigma}}
      &{\spH{\faidlof|\sigma'|/-1*}} \ar[r]
      &{\dotsm}
    }
  \end{gathered}
\end{equation}
where the rows are exact and $\delta$ is the connecting morphism
considered in Proposition
\ref{prop:res-formula-dbar-exact-dot-harmonic}.
The injectivity theorem in \cite{Chan&Choi&Matsumura_injectivity}
implies the following proposition.

\begin{prop} \label{prop:ker-tau-reduction}
  Under the assumption that $\ibddbar\bphi \geq 0$ on $X$, one has
  \begin{equation*}
    \ker\tau^{\sigma'}_{\sigma}
    = \ker\tau^{\sigma+1}_{\sigma} = \im \delta
    \isom \im\deltaH \; .
  \end{equation*}
\end{prop}

\begin{proof}
  The equality $\ker\tau^{\sigma+1}_{\sigma} = \im\delta$ follows from
  the exactness of the first row of
  \eqref{eq:reduce-to-residl-comm-diagram} and $\im\delta \isom
  \im\deltaH$ is a consequence of the validity of the harmonic theory,
  namely, $\spH/\bullet/{\residlof|\star|*} \isom
  \spHarm/\bullet/{\residlof|\star|*}$.
  The remaining equality $\ker\tau^{\sigma'}_{\sigma}
  =\ker\tau^{\sigma+1}_{\sigma}$ follows from
  \cite{Chan&Choi&Matsumura_injectivity}*{the proof of Thm.~3.4.1} and
  \cite{Chan&Choi&Matsumura_injectivity}*{Remark 3.4.2} with suitable adjustments.
  For the sake of clarity, a direct argument using Proposition
  \ref{prop:res-formula-dbar-exact-dot-harmonic} is given here.

  Note that $\ker\tau^{\sigma+1}_{\sigma} \subset
  \ker\tau^{\sigma'}_{\sigma}$ follows from the commutativity of
  \eqref{eq:reduce-to-residl-comm-diagram}.
  It suffices to show the reverse inclusion.
  After identifying the harmonic spaces with the cohomology groups,
  one can pick any $u =\paren{u_p}_{p\in \Iset} \in
  \ker\tau^{\sigma'}_{\sigma}$ such that $u$ is orthogonal to
  $\ker\tau^{\sigma+1}_{\sigma} = \im\deltaH$ in $\spHarm{\residlof*}$.
  The goal is then to show that $u = 0$
  (i.e.~$\ker\tau^{\sigma'}_{\sigma} \cap
  \paren{\ker\tau^{\sigma+1}_{\sigma}}^\perp = 0$), which will
  complete the proof.
  
  Proposition \ref{prop:res-formula-dbar-exact-dot-harmonic}
  guarantees that $\HRes(u) \in \spHarm/q-1/{\residlof+1*}$ and
  \begin{equation*}
    \sigma_+\iinner{w}{\HRes(u)}_{\lcc+1<X^\circ>'} = \iinner{\deltaH w}{u}_{\lcc<X^\circ>'}
    \overset{\alert[Gray]{\mathllap{\im\,}\deltaH \perp u}}= 0
    \quad\text{ for all } w \in \spHarm/q-1/{\residlof+1*} \; .
  \end{equation*}
  This implies that $\HRes(u) = 0$.
  
  On the other hand, $u \in \ker\tau^{\sigma'}_{\sigma} =\im\delta'$
  means that the class of each $u_p = \res u_{\lcS}$ can be represented
  (via \eqref{eq:Cech-Dolbeault-isom}) by a \v Cech coboundary
  $\set{\alpha_{p;\idx 0.q}}_{\idx 0,q \in I} = \set{\frac{\paren{\delta
      \rs*\gamma_p}_{\idx 0.q}}{\sect\ps_p}}_{\idx 0,q\in I}$, and the computation
  in the proof of Proposition
  \ref{prop:res-formula-dbar-exact-dot-harmonic}
  (cf.~\cite{Chan&Choi&Matsumura_injectivity}*{Step II of the proof of
    Thm.~3.4.1} or
  \cite{Chan&Choi&Matsumura_injectivity-II}*{Step 3 of the proof of
    Thm.~5.2.1}) leads to 
  \begin{equation*}
    \norm u_{\lcc<X^\circ>'}^2
    =
    \smashoperator[r]{\sum_{b\in \Iset+1}} \sigma_+
    \iinner{
      \frac{v_{b;(\infty)}}{\sect\ps_b}\:
    }{\HRes(u)_b}_{\lcS+1[b]'}
    \overset{\mathclap{\alert[Gray]{\HRes(u) = 0}}}= \quad 0
    \; ,
  \end{equation*}
  where the notations follow those in the proof of Proposition
  \ref{prop:res-formula-dbar-exact-dot-harmonic}, with the only change that the local
  sections $\gamma_{\idx 1.q}$ take values in $\aidlof|\sigma'|*$
  instead of $\aidlof+1*$.
  The proof is then complete.
\end{proof}




}

\subsection{Existence of extension}
\label{sec:global-extension}



The main theorem is proved in this section.
Recall that, under the snc assumption, there is a decomposition
$\vphi_L+\psi =\bphi +\phi_{S_0} +\phi_S$ such that
$\bphi^{-1}(-\infty)$ contains no lc centres of $(X,S)$.
For every $\sigma = 0, 1, \dots, n$ and $p \in \Iset$ (with
$\Iset|0| := \set 0$ being a singleton and $\lcS|0|[] :=\lcS|0|[0] :=
X$), let $\Psi\ps \leq -1$ be a quasi-psh function with neat analytic
singularities on $\lcS$ (thus $\Psi\ps \not\equiv -\infty$ on $\lcS$)
with its polar set being an snc divisor such that
it decomposes as in \eqref{eq:decompose-psi_S+P} into
\begin{equation*}
  \Psi\ps = \psi\ps. + \psi_{\mu_p \cdot B} \quad\text{ on } \lcS \; ,
\end{equation*}
where $\psi\ps. := \psi_{\nu_p \cdot S\ps}$ and 
\begin{equation*}
  \psi^{-1}\ps.(-\infty) \cap \lcS = \Diff_{p} S
  \quad\text{ and }\quad
  \psi^{-1}_{\mu_p\cdot B}(-\infty) \cap \lcS \subset B \cap \lcS
  \quad\text{ (as sets)}
  \; .
\end{equation*}
%

\begin{thm} \label{thm:global-extension}
  For a given $\sigma = 0, 1, \dots, \sigma_{\mlc}$, suppose that
  there exists a constant $\lambda_0 > 0$ such that
  \begin{equation} \label{eq:curvature-assumption}
    \ibddbar\paren{\bphi +\lambda \Psi\ps} \geq 0
    \quad\text{ on } \lcS \text{ for all } \lambda \in [0,\lambda_0]
    \text{ and for each } p \in \Iset \; .
  \end{equation}
  Then, the short exact sequence
  $\xymatrix@1{
    *+<1ex>{0} \ar[r]
    &*+<1ex>{\residlof*} \ar[r]
    &*+<1ex>{\faidlof|\sigma'|/-1*} \ar[r]
    &*+<1ex>{\faidlof|\sigma'|*} \ar[r]
    &*+<1ex>{0}
  }$ for any $\sigma' \geq \sigma$ induces a long exact sequence which
  splits into short exact sequences 
  \begin{equation*}
    \xymatrix{
      {0} \ar[r]
      &{\spH{\residlof*}} \ar[r]^-{\tau_\sigma^{\sigma'}}
      &{\spH{\faidlof|\sigma'|/-1*}} \ar[r]
      &{\spH{\faidlof|\sigma'|/*}} \ar[r]
      &{0}
    } \quad\text{ for all } q \geq 0 \; .
  \end{equation*}
  In particular, for every $\sigma' > \sigma$, each $f \in
  \spH/0/{\faidlof|\sigma'|*}$ has an ``extension'' $F_{\sigma -1} \in
  \spH/0/{\faidlof|\sigma'|/-1*}$ such that $f \equiv F_{\sigma -1} \mod
  \aidlof*$.

  If \eqref{eq:curvature-assumption} holds true for every integer $\sigma \geq
  0$, it follows by induction that, for every $\sigma' \geq 0$, every $f \in
  \spH/0/{\faidlof|\sigma'|/|\sigma'-1|*}
  \isom \spH/0/{\residlof|\sigma'|*}$ (a section on $\lcc|\sigma'|'$)
  has an extension $F_{-1} \in \spH/0/{\faidlof|\sigma'|/|-1|*}
  =\spH/0/{\aidlof|\sigma'|*}$ (a section on $X$) such that $f \equiv
  F_{-1} \mod \aidlof|\sigma' -1|*$.
\end{thm}

\begin{proof}
  \setDefaultMetric{\rs\omega}

  It suffices to show that $\ker\tau^{\sigma'}_\sigma = 0$ for every
  $\sigma' \geq \sigma$ (and for every $q \geq 0$).
  Note that the equality
  \begin{equation*}
    K_X \otimes L \otimes \residlof
    = \bigoplus_{p\in\Iset} K_{\lcS} \otimes \underbrace{
      \parres{L \otimes S_0^{-1}
      \otimes S^{-1}}_{\lcS} 
      }_{=: \: \res{L'}_{\lcS}} \otimes
      \mtidlof<\lcS>{\bphi} \; ,
  \end{equation*}
  where $\res{L'}_{\lcS}$ is a line bundle on $\lcS$ equipped with the
  potential $\res{\bphi}_{\lcS}$, holds true even for $\sigma = 0$,
  as, by convention, $\residlof|0| =S_0^{-1} \otimes S^{-1}
  \otimes \mtidlof<X>{\bphi} \isom
  \mtidlof<X>{\bphi+\phi_{S_0}+\phi_S} =\aidlof|0|<X>
  =\faidlof|0|/|-1|*$.
  It can be seen below that the proof for $\ker\tau^{\sigma'}_\sigma =
  0$ is formally the same as the proof for $\ker\tau^{\sigma'}_0 =
  0$.
  In what follows, the proof for the case $\sigma = 0$ is given first
  for the benefits of readers.
  The proof for the general statement is given afterwards.

  Let $\rs\omega$ be a \emph{complete} metric on $X^\circ := X
  \setminus B$.
  Write $\lcS' := \lcS \cap X^\circ$ for any $\sigma \geq 0$ and any
  $p \in \Iset$.
  Take any \emph{harmonic form} $u \in \Harm{K_X \otimes
    L'},{\bphi} \isom \spH{\residlof|0|*}$ on $X^\circ$ such that
  $\tau^{\sigma'}_0(u) = 0$.
  The positivity assumption $\ibddbar\bphi \geq 0$, together with
  Proposition \ref{prop:consequence-of-positivity}, implies that
  \begin{equation*}
    \nabla^{(0,1)}u = 0 \quad\text{ and }\quad
    \idxup{\ibddbar\bphi}\ptinner u u_{\bphi,\rs\omega} = 0
    \quad\text{ on } X^\circ \; ,
  \end{equation*}
  which, according to Proposition \ref{prop:consequence-of-tBK}, yields
  \begin{align*}
    \int_{X^\circ} \idxup{\ibddbar\sm\vphi\ps._0}\ptinner u u_{\bphi,\rs\omega}
    &=\lim_{\eps \tendsto 0^+} \eps \int_{\mathrlap{X^\circ}} \;\; \frac{
      \abs{\idxup{\diff\psi\ps._0}* . u}_{\bphi,\rs\omega}^2
    }{\abs{\psi\ps._0}^{1+\eps}} \geq 0 \; ,
    \\
    \int_{X^\circ} \idxup{\ibddbar\sm\vphi_{\mu_0\cdot B}} \ptinner u u_{\bphi,\rs\omega}
    &=\lim_{\eps \tendsto 0^+} \eps \int_{\mathrlap{X^\circ}} \;\; \frac{
      \abs{\idxup{\diff\psi_{\mu_0\cdot B}}* . u}_{\bphi,\rs\omega}^2
    }{\abs{\psi_{\mu_0\cdot B}}^{1+\eps}} \geq 0 \; .
  \end{align*}
  Note that $\ibddbar\Psi\ps_0 = \ibddbar\psi\ps._0
  +\ibddbar\psi_{\mu_0\cdot B} = -\ibddbar\sm\vphi\ps._0
  -\ibddbar\sm\vphi_{\mu_0\cdot B}$ on $X^\circ \setminus S$.
  The assumption \eqref{eq:curvature-assumption}${}_0$ (the assumption
  with $\sigma = 0$) together with $\idxup{\ibddbar\bphi}\ptinner u
  u_{\bphi,\rs\omega} = 0$ implies that the sum of the left-hand sides
  of both equations above is
  \emph{non-positive}.
  That in turn implies, in particular, that
  \begin{equation*}
    0 = \lim_{\eps \tendsto 0^+} \eps \int_{\mathrlap{X^\circ}} \;\; \frac{
      \abs{\idxup{\diff\psi\ps._0}* . u}_{\bphi,\rs\omega}^2
    }{\abs{\psi\ps._0}^{1+\eps}}
    \;\overset{\text{Prop.~\ref{prop:residue-formula-classical-kernel}}}=\;
    \sum_{b \in \Iset|1|} \frac{\pi}{\nu_{b:0}} \int_{\lcS|1|[b]'}
    \abs{\PRes[\lcS|1|[b]](\idxup{\diff\psi\ps._0}
      . u)}_{\bphi,\rs\omega}^2 \; ,
  \end{equation*}
  and therefore $\HRes(u) = \paren{
    \PRes[\lcS|1|[b]](\idxup{\diff\psi\ps._0} . u)}_{b \in \Iset|1|} =
  0$. 
  Proposition \ref{prop:ker-tau-reduction} implies that $u \in
  \ker\tau^{\sigma'}_0 = \im\deltaH$, thus $u = \deltaH w$ for
  some $w \in \spHarm/q-1/{\residlof|1|*}$.
  Proposition \ref{prop:res-formula-dbar-exact-dot-harmonic} then yields
  \begin{equation*}
    \norm u_{X^\circ}^2
    =\iinner{\deltaH w}{u}_{X^\circ}
    =\iinner{w}{\HRes(u)}_{\lcc|1|<X^\circ>'}
    =0 \; .
  \end{equation*}
  Therefore, $u = 0$ and, consequently, $\ker\tau^{\sigma'}_0 = 0$.

  For the general case, take $u = \paren{u_p}_{p \in \Iset} \in
  \bigoplus_{p\in\Iset} \Harm<\lcS>{\logKX<\lcS>[L']},{\bphi} \isom
  \spH{\residlof*}$ such that $\tau^{\sigma'}_\sigma (u) = 0$.
  The same argument yields the equations $\nabla^{(0,1)}u_p = 0$ and
  $\idxup{\ibddbar\bphi} \ptinner{u_p}{u_p}_{\bphi,\rs\omega} = 0$ on
  $\lcS'$, and therefore, by Proposition \ref{prop:consequence-of-tBK},
  \begin{align*}
    \int_{\lcS'} \idxup{\ibddbar\sm\vphi\ps.}\ptinner{u_p}{u_p}_{\bphi,\rs\omega}
    &=\lim_{\eps \tendsto 0^+} \eps \int_{\mathrlap{\lcS'}} \;\; \frac{
      \abs{\idxup{\diff\psi\ps.}* . u_p}_{\bphi,\rs\omega}^2
    }{\abs{\psi\ps.}^{1+\eps}} \geq 0 \; ,
    \\
    \int_{\lcS'} \idxup{\ibddbar\sm\vphi_{\mu_p\cdot B}}\ptinner{u_p}{u_p}_{\bphi,\rs\omega}
    &=\lim_{\eps \tendsto 0^+} \eps \int_{\mathrlap{\lcS'}} \;\; \frac{
      \abs{\idxup{\diff\psi_{\mu_p\cdot B}}* . u_p}_{\bphi,\rs\omega}^2
    }{\abs{\psi_{\mu_p\cdot B}}^{1+\eps}} \geq 0 \; .
  \end{align*}
  The assumption \eqref{eq:curvature-assumption}${}_{\sigma}$
  then implies that 
  \begin{equation*}
    0 = \lim_{\eps \tendsto 0^+} \eps \int_{\mathrlap{\lcS'}} \;\; \frac{
      \abs{\idxup{\diff\psi\ps.}* . u_p}_{\bphi,\rs\omega}^2
    }{\abs{\psi\ps.}^{1+\eps}}
    \;\overset{\text{Prop.~\ref{prop:residue-formula-classical-kernel}}}=\;
    \smashoperator{\sum_{b \in \Iset+1 \colon \lcS+1[b] \subset \lcS}}
    \quad\;\,
    \frac{\pi}{\nu_{b:p}} \int_{\lcS+1[b]'}
    \abs{\PRes[\lcS+1[b] | \lcS](\idxup{\diff\psi\ps.}
      . u_p)}_{\bphi,\rs\omega}^2 \; ,
  \end{equation*}
  and therefore $\PRes[\lcS+1[b] | \lcS](\idxup{\diff\psi\ps.}
  . u_p) = 0$ for every $b \in \Iset+1$ and $p \in\Iset$ such that
  $\lcS+1[b] \subset \lcS$.
  Proposition \ref{prop:ker-tau-reduction} says that $u \in \ker\tau^{\sigma'}_\sigma = \im\deltaH$.
  Proposition \ref{prop:res-formula-dbar-exact-dot-harmonic} then gives the equation 
  \begin{equation*}
    \norm u_{\lcc<X^\circ>'}^2
    =\iinner{\deltaH w}{u}_{\lcc<X^\circ>'}
    =\sigma_+ \iinner{w}{\HRes(u)}_{\lcc+1<X^\circ>'}
    \; ,
  \end{equation*}
  where $u = \deltaH w$ for some $w \in \spHarm/q-1/{\residlof+1*}$ and
  $\HRes(u) = \paren{\HRes(u)_b}_{b\in \Iset+1}$ with
  \begin{equation*}
    \HRes(u)_b
    =\quad\;\;
    \smashoperator{\sum_{p \in \Iset \colon \lcS+1[b] \subset \lcS}}
    \: \sgn{b:p} \:
    \PRes[\lcS+1[b] | \lcS](\idxup{\diff\psi\ps.} . u_p) =0 \; .
  \end{equation*}
  This again implies that $u = 0$ and thus $\ker\tau^{\sigma'}_\sigma
  = 0$, which completes the proof.
\end{proof}

\begin{remark} \label{rem:curv-assumption-on-X_only}
  It can be seen from the proof that the coefficients $\nu_p$ and
  $\mu_p$ in $\Psi\ps$ need not be related to the coefficients in
  $\psi$.
  In practice, it may be useful (for example, for handling the dlt conjecture
  \cite{DHP}*{Conj.~1.3}) if one can construct $\Psi\ps$ satisfying
  \eqref{eq:curvature-assumption}${}_{\sigma}$ on $\lcS$ for $\sigma
  \geq 1$ from $\psi$ when $\Psi\ps_0 := \psi$ satisfies
  \eqref{eq:curvature-assumption}${}_0$ on $X$.
  At the time of writing, it is unknown to the author whether there is
  a general procedure such that this can be achieved.
\end{remark}

\begin{remark} \label{rem:weaker-assumptions}
  Note also that the positivity assumption
  \eqref{eq:curvature-assumption} is used to show that $\HRes(u) = 0$
  by proving the vanishing of each of its summands $\PRes[\lcS+1[b] |
  \lcS](\dots)$.
  It would be interesting to see if the assumption
  \eqref{eq:curvature-assumption} can be weakened so that $\HRes(u)$
  vanishes without resorting to the vanishing of all summands.
  Indeed, in the proof of the injectivity theorems in
  \cites{Chan&Choi&Matsumura_injectivity,
    Chan&Choi&Matsumura_injectivity-II}, this is the case under a weaker
  positivity assumption (only $\ibddbar\bphi \geq 0$), but the section
  $u$ there has to satisfy an assumption which is irrelevant
  to the extension problem.
\end{remark}






\begin{bibdiv}
  \begin{biblist}
    \IfFileExists{references.ltb}{
      \bibselect{references}
    }{
      \bib{Berndtsson_OT-extension}{article}{
  author={Berndtsson, Bo},
  title={The extension theorem of Ohsawa-Takegoshi and the theorem of Donnelly-Fefferman},
  language={English, with English and French summaries},
  journal={Ann. Inst. Fourier (Grenoble)},
  volume={46},
  date={1996},
  number={4},
  pages={1083--1094},
  issn={0373-0956},
  review={\MR {1415958}},
}

\bib{Berndtsson_ext-form}{article}{
  author={Berndtsson, Bo},
  title={$L^2$-extension of $\overline {\partial }$-closed form},
  journal={Illinois J. Math.},
  volume={56},
  date={2012},
  number={1},
  pages={21--31 (2013)},
  issn={0019-2082},
  review={\MR {3117015}},
}

\bib{Berndtsson&Lempert}{article}{
  author={Berndtsson, Bo},
  author={Lempert, L\'{a}szl\'{o}},
  title={A proof of the Ohsawa-Takegoshi theorem with sharp estimates},
  journal={J. Math. Soc. Japan},
  volume={68},
  date={2016},
  number={4},
  pages={1461--1472},
  issn={0025-5645},
  review={\MR {3564439}},
  doi={10.2969/jmsj/06841461},
}

\bib{Blocki_Suita-conj}{article}{
  author={B\l ocki, Zbigniew},
  title={Suita conjecture and the Ohsawa-Takegoshi extension theorem},
  journal={Invent. Math.},
  volume={193},
  date={2013},
  number={1},
  pages={149--158},
  issn={0020-9910},
  review={\MR {3069114}},
  doi={10.1007/s00222-012-0423-2},
}

\bib{Cao&Demailly&Matsumura}{article}{
  author={Cao, JunYan},
  author={Demailly, Jean-Pierre},
  author={Matsumura, Shinichi},
  title={A general extension theorem for cohomology classes on non reduced analytic subspaces},
  journal={Sci. China Math.},
  volume={60},
  date={2017},
  number={6},
  pages={949--962},
  issn={1674-7283},
  review={\MR {3647124}},
  doi={10.1007/s11425-017-9066-0},
}

\bib{Cao&Paun_OT-ext}{article}{
  author={Cao, Junyan},
  author={P\u aun, Mihai},
  author={Berndtsson, Bo},
  title={On the Ohsawa--Takegoshi extension theorem},
  journal={J. Geom. Anal.},
  volume={34},
  date={2024},
  number={1},
  pages={Paper No. 25, 47},
  issn={1050-6926},
  review={\MR {4668049}},
  doi={10.1007/s12220-023-01466-9},
  arxiv={2002.04968 [math.CV]},
}

\bib{Chan_on-L2-ext-with-lc-measures}{article}{
  author={Chan, Tsz On Mario},
  title={On an $L^2$ extension theorem from log-canonical centres with log-canonical measures},
  journal={Math. Z.},
  volume={301},
  date={2022},
  number={2},
  pages={1695--1717},
  issn={0025-5874},
  review={\MR {4418335}},
  doi={10.1007/s00209-021-02890-9},
  eprint={https://rdcu.be/cFDPA},
  arxiv={2008.03019 [math.CV]},
  note={Numbering of cited sections and theorems follows the arXiv version},
}

\bib{Chan_adjoint-ideal-nas}{article}{
  author={Chan, Tsz On Mario},
  title={A new definition of analytic adjoint ideal sheaves via the residue functions of log-canonical measures I},
  journal={J. Geom. Anal.},
  volume={33},
  date={2023},
  number={9},
  pages={Paper No. 279, 68},
  issn={1050-6926},
  review={\MR {4605571}},
  doi={10.1007/s12220-023-01314-w},
  eprint={https://rdcu.be/deUDt},
  arxiv={2111.05006 [math.CV]},
}

\bib{Chan_Residue-fct-proceedings}{article}{
  author={Chan, Tsz On Mario},
  title={Residue functions and extension problems},
  arxiv={2211.00885 [math.CV]},
  date={2022},
  note={To appear in Proceedings of CCGA2022 and KSCV14},
}

\bib{Chan&Choi_ext-with-lcv-codim-1}{article}{
  author={Chan, Tsz On Mario},
  author={Choi, Young-Jun},
  title={Extension with log-canonical measures and an improvement to the plt extension of Demailly-Hacon-P\u {a}un},
  journal={Math. Ann.},
  volume={383},
  date={2022},
  number={3-4},
  pages={943--997},
  issn={0025-5831},
  review={\MR {4458394}},
  doi={10.1007/s00208-021-02152-3},
  eprint={https://rdcu.be/cn5N6},
  arxiv={1912.08076 [math.CV]},
}

\bib{Chan&Choi_injectivity-I}{article}{
  author={Chan, Tsz On Mario},
  author={Choi, Young-Jun},
  title={On an injectivity theorem for log-canonical pairs with analytic adjoint ideal sheaves},
  journal={Trans. Amer. Math. Soc.},
  volume={376},
  number={12},
  pages={8337--8381},
  issn={0002-9947},
  review={\MR {4669299}},
  doi={10.1090/tran/8935},
  arxiv={2205.06954 [math.CV]},
  date={2023},
}

\bib{Chan&Choi_injectivity-proceedings}{article}{
  author={Chan, Tsz On Mario},
  author={Choi, Young-Jun},
  title={An application of adjoint ideal sheaves to injectivity and extension theorems},
  conference={ title={Convex and complex: perspectives on positivity in geometry}, },
  book={ series={Contemp. Math.}, volume={810}, publisher={Amer. Math. Soc., Providence, RI}, },
  isbn={978-1-4704-7338-9},
  isbn={[9781470478612]},
  date={[2025] \copyright 2025},
  pages={83--97},
  review={\MR {4853194}},
  arxiv={2306.00670 [math.CV]},
  doi={10.1090/conm/810/16208},
}

\bib{Chan&Choi&Matsumura_injectivity}{article}{
  author={Chan, Tsz On Mario},
  author={Choi, Young-Jun},
  author={Matsumura, Shinichi},
  title={An injectivity theorem on snc compact K\"ahler spaces: an application of the theory of harmonic integrals on log-canonical centers via adjoint ideal sheaves},
  arxiv={2307.12025 [math.CV]},
  date={2023},
}

\bib{Chan&Choi&Matsumura_injectivity-II}{article}{
  author={Chan, Tsz On Mario},
  author={Choi, Young-Jun},
  author={Matsumura, Shin-ichi},
  title={Injectivity theorems for higher direct images under proper K\"ahler morphisms on snc spaces},
  arxiv={arXiv:2409.14100 [math.CV]},
  date={2024},
}

\bib{Demailly_on_OTM-extension}{article}{
  author={Demailly, Jean-Pierre},
  title={On the Ohsawa-Takegoshi-Manivel $L^2$ extension theorem},
  language={English, with English and French summaries},
  conference={ title={Complex analysis and geometry}, address={Paris}, date={1997}, },
  book={ series={Progr. Math.}, volume={188}, publisher={Birkh\"{a}user, Basel}, },
  date={2000},
  pages={47--82},
  review={\MR {1782659}},
}

\bib{Demailly_multiplier-ideal-sheaves}{article}{
  author={Demailly, Jean-Pierre},
  title={Multiplier ideal sheaves and analytic methods in algebraic geometry},
  conference={ title={School on Vanishing Theorems and Effective Results in Algebraic Geometry}, address={Trieste}, date={2000}, },
  book={ series={ICTP Lect. Notes}, volume={6}, publisher={Abdus Salam Int. Cent. Theoret. Phys., Trieste}, },
  date={2001},
  pages={1--148},
  review={\MR {1919457}},
}

\bib{Demailly}{webpage}{
  author={Demailly, Jean-Pierre},
  title={Complex analytic and differential geometry},
  note={OpenContent Book},
  url={https://www-fourier.ujf-grenoble.fr/~demailly/manuscripts/agbook.pdf},
  date={2012},
}

\bib{Demailly_extension}{article}{
  author={Demailly, Jean-Pierre},
  title={Extension of holomorphic functions defined on non reduced analytic subvarieties},
  conference={ title={The legacy of Bernhard Riemann after one hundred and fifty years. Vol. I}, },
  book={ series={Adv. Lect. Math. (ALM)}, volume={35}, publisher={Int. Press, Somerville, MA}, },
  date={2016},
  pages={191--222},
  review={\MR {3525916}},
  arxiv={1510.05230 [math.CV]},
}

\bib{DHP}{article}{
  author={Demailly, Jean-Pierre},
  author={Hacon, Christopher D.},
  author={P\u {a}un, Mihai},
  title={Extension theorems, non-vanishing and the existence of good minimal models},
  journal={Acta Math.},
  volume={210},
  date={2013},
  number={2},
  pages={203--259},
  issn={0001-5962},
  review={\MR {3070567}},
  doi={10.1007/s11511-013-0094-x},
}

\bib{Fujino_injectivity-II}{article}{
  author={Fujino, Osamu},
  title={A transcendental approach to Koll\'{a}r's injectivity theorem II},
  journal={J. Reine Angew. Math.},
  volume={681},
  date={2013},
  pages={149--174},
  issn={0075-4102},
  review={\MR {3181493}},
  doi={10.1515/crelle-2012-0036},
}

\bib{Guan&Zhou_optimal-L2-estimate}{article}{
  author={Guan, Qi'an},
  author={Zhou, Xiangyu},
  title={A solution of an $L^2$ extension problem with an optimal estimate and applications},
  journal={Ann.~of Math.~(2)},
  volume={181},
  date={2015},
  number={3},
  pages={1139--1208},
  issn={0003-486X},
  review={\MR {3296822}},
  doi={10.4007/annals.2015.181.3.6},
}

\bib{Guan&Zhou_openness}{article}{
  author={Guan, Qi'an},
  author={Zhou, Xiangyu},
  title={A proof of Demailly's strong openness conjecture},
  journal={Ann. of Math. (2)},
  volume={182},
  date={2015},
  number={2},
  pages={605--616},
  issn={0003-486X},
  review={\MR {3418526}},
  doi={10.4007/annals.2015.182.2.5},
}

\bib{Guan&Zhou_effective_openness}{article}{
  author={Guan, Qi'an},
  author={Zhou, Xiangyu},
  title={Effectiveness of Demailly's strong openness conjecture and related problems},
  journal={Invent. Math.},
  volume={202},
  date={2015},
  number={2},
  pages={635--676},
  issn={0020-9910},
  review={\MR {3418242}},
  doi={10.1007/s00222-014-0575-3},
}

\bib{Guan&Zhou&Zhu_OTExt}{article}{
  author={Guan, Qi'an},
  author={Zhou, Xiangyu},
  author={Zhu, Langfeng},
  title={On the Ohsawa--Takegoshi $L^2$ extension theorem and the Bochner--Kodaira identity with non-smooth twist factor},
  journal={J.~Math.~Pure.~Appl.},
  volume={97},
  number={6},
  pages={579--601},
  date={2012},
}

\bib{Hiep_openness}{article}{
  author={Hiep, Pham Hoang},
  title={The weighted log canonical threshold},
  language={English, with English and French summaries},
  journal={C. R. Math. Acad. Sci. Paris},
  volume={352},
  date={2014},
  number={4},
  pages={283--288},
  issn={1631-073X},
  review={\MR {3186914}},
  doi={10.1016/j.crma.2014.02.010},
}

\bib{KimDano_lc-extension}{article}{
  author={Kim, Dano},
  title={$L^2$ extension of adjoint line bundle sections},
  language={English, with English and French summaries},
  journal={Ann. Inst. Fourier (Grenoble)},
  volume={60},
  date={2010},
  number={4},
  pages={1435--1477},
  issn={0373-0956},
  review={\MR {2722247}},
}

\bib{KimDano-L2-ext-for-lc}{article}{
  author={Kim, Dano},
  title={$L^{2}$ extension of holomorphic functions for log canonical pairs},
  language={English, with English and French summaries},
  journal={J. Math. Pures Appl. (9)},
  volume={177},
  date={2023},
  pages={198--213},
  issn={0021-7824},
  review={\MR {4629755}},
  doi={10.1016/j.matpur.2023.06.013},
  arxiv={2108.11934 [math.CV]},
}

\bib{Kollar_Sing-of-MMP}{book}{
  author={Koll\'{a}r, J\'{a}nos},
  title={Singularities of the minimal model program},
  series={Cambridge Tracts in Mathematics},
  volume={200},
  note={With a collaboration of S\'{a}ndor Kov\'{a}cs},
  publisher={Cambridge University Press, Cambridge},
  date={2013},
  pages={x+370},
  isbn={978-1-107-03534-8},
  review={\MR {3057950}},
  doi={10.1017/CBO9781139547895},
}

\bib{Lazarsfeld_book-II}{book}{
  author={Lazarsfeld, Robert},
  title={Positivity in algebraic geometry. II},
  series={Ergebnisse der Mathematik und ihrer Grenzgebiete. 3. Folge. A Series of Modern Surveys in Mathematics [Results in Mathematics and Related Areas. 3rd Series. A Series of Modern Surveys in Mathematics]},
  volume={49},
  note={Positivity for vector bundles, and multiplier ideals},
  publisher={Springer-Verlag, Berlin},
  date={2004},
  pages={xviii+385},
  isbn={3-540-22534-X},
  review={\MR {2095472}},
  doi={10.1007/978-3-642-18808-4},
}

\bib{Manivel}{article}{
  author={Manivel, Laurent},
  title={Un th\'{e}or\`eme de prolongement $L^2$ de sections holomorphes d'un fibr\'{e} hermitien},
  language={French},
  journal={Math. Z.},
  volume={212},
  date={1993},
  number={1},
  pages={107--122},
  issn={0025-5874},
  review={\MR {1200166}},
  doi={10.1007/BF02571643},
}

\bib{Matsumura_injectivity}{article}{
  author={Matsumura, Shinichi},
  title={An injectivity theorem with multiplier ideal sheaves of singular metrics with transcendental singularities},
  journal={J. Algebraic Geom.},
  volume={27},
  date={2018},
  number={2},
  pages={305--337},
  issn={1056-3911},
  review={\MR {3764278}},
  doi={10.1090/jag/687},
  arxiv={1308.2033 [math.CV]},
}

\bib{Matsumura_injectivity-lc}{article}{
  author={Matsumura, Shinichi},
  title={A transcendental approach to injectivity theorem for log canonical pairs},
  journal={Ann. Sc. Norm. Super. Pisa Cl. Sci. (5)},
  volume={19},
  date={2019},
  number={1},
  pages={311--334},
  issn={0391-173X},
  review={\MR {3923849}},
}

\bib{Matsumura_injectivity-Kaehler}{article}{
  author={Matsumura, Shinichi},
  title={Injectivity theorems with multiplier ideal sheaves for higher direct images under K\"{a}hler morphisms},
  journal={Algebr. Geom.},
  volume={9},
  date={2022},
  number={2},
  pages={122--158},
  issn={2313-1691},
  review={\MR {4429015}},
  doi={10.14231/ag-2022-005},
  arxiv={1607.05554v2 [math.CV]},
}

\bib{McNeal&Varolin_adjunction}{article}{
  author={McNeal, Jeffery D.},
  author={Varolin, Dror},
  title={Analytic inversion of adjunction: $L^2$ extension theorems with gain},
  language={English, with English and French summaries},
  journal={Ann. Inst. Fourier (Grenoble)},
  volume={57},
  date={2007},
  number={3},
  pages={703--718},
  issn={0373-0956},
  review={\MR {2336826}},
}

\bib{McNeal&Varolin_L2-estimates}{article}{
  author={McNeal, Jeffery D.},
  author={Varolin, Dror},
  title={$L^2$ estimates for the $\overline \partial $ operator},
  journal={Bull. Math. Sci.},
  volume={5},
  date={2015},
  number={2},
  pages={179--249},
  issn={1664-3607},
  review={\MR {3354033}},
  doi={10.1007/s13373-015-0068-8},
}

\bib{Ohsawa-II}{article}{
  author={Ohsawa, Takeo},
  title={On the extension of $L^2$ holomorphic functions. II},
  journal={Publ. Res. Inst. Math. Sci.},
  volume={24},
  date={1988},
  number={2},
  pages={265--275},
  issn={0034-5318},
  review={\MR {944862}},
  doi={10.2977/prims/1195175200},
}

\bib{Ohsawa-III}{article}{
  author={Ohsawa, Takeo},
  title={On the extension of $L^2$ holomorphic functions. III. Negligible weights},
  journal={Math. Z.},
  volume={219},
  date={1995},
  number={2},
  pages={215--225},
  issn={0025-5874},
  review={\MR {1337216}},
  doi={10.1007/BF02572360},
}

\bib{Ohsawa&Takegoshi-I}{article}{
  author={Ohsawa, Takeo},
  author={Takegoshi, Kensh\={o}},
  title={On the extension of $L^2$ holomorphic functions},
  journal={Math. Z.},
  volume={195},
  date={1987},
  number={2},
  pages={197--204},
  issn={0025-5874},
  review={\MR {892051}},
  doi={10.1007/BF01166457},
}

\bib{Paun_inv-plurigenera}{article}{
  author={P\u aun, Mihai},
  title={Siu's invariance of plurigenera: a one-tower proof},
  journal={J.~Diff.~Geom.},
  volume={76},
  pages={485--493},
  date={2007},
}

\bib{Siu_inv_plurigenera2}{article}{
  author={Siu, Yum-Tong},
  title={Extension of twisted pluricanonical sections with plurisubharmonic weight and invariance of semipositively twisted plurigenera for manifolds not necessarily of general type},
  conference={ title={Complex geometry}, address={G\"{o}ttingen}, date={2000}, },
  book={ publisher={Springer, Berlin}, },
  date={2002},
  pages={223--277},
  review={\MR {1922108}},
}

\bib{Takayama_adj-ideal}{article}{
  author={Takayama, Shigeharu},
  title={Pluricanonical systems on algebraic varieties of general type},
  journal={Invent. Math.},
  volume={165},
  date={2006},
  number={3},
  pages={551--587},
  issn={0020-9910},
  review={\MR {2242627}},
  doi={10.1007/s00222-006-0503-2},
}


    }
  \end{biblist}
\end{bibdiv}
\end{document}